\documentclass[11pt]{amsart} 

\usepackage{fullpage}
\usepackage{amssymb}
\usepackage{amsmath}
\usepackage{hyperref}
\newcommand{\arxiv}[1]{\href{http://www.arXiv.org/abs/#1}{arXiv:#1}}
\usepackage{graphicx,color,graphics}
\usepackage{latexsym} 
\usepackage{paralist}
\usepackage{listings}

\newtheorem{theorem}{Theorem}
\newtheorem{lemma}[theorem]{Lemma} 
\newtheorem{definition}[theorem]{Definition}
\newtheorem{proposition}[theorem]{Proposition}
\newtheorem{propdef}[theorem]{Proposition-Definition}
\newtheorem{corollary}[theorem]{Corollary}
\newtheorem{remark}{Remark}
\newtheorem{example}{Example}
\newtheorem{myalgorithm}{Algorithm}

\newcommand{\mysup}{\mathop{\text{\Large$\vee$}}}
\newcommand{\myinf}{\mathop{\text{\Large$\wedge$}}}
\newcommand{\objective}{g}
\newcommand{\spectral}{\phi}
\def\ltr{\text{``}}
\def\rtr{\text{''}}

\def\digr{{\mathcal D}}
\def\Bipdig{{\mathcal G}}
\def\RR{\mathbb{R}_{\max}}
\def\Germ{\mathbb{G}_{\max}}
\def\ZGerm{-\infty_\mathbb{G}}
\def\Rmax{\mathbb{R}_{\max}}
\def\Rmin{\mathbb{R}_{\min}}
\def\RRbar{\overline{\mathbb{R}}}
\def\RRmax{\RRbar_{\max}}

\def\R{\mathbb{R}}

\def\diez{\sharp}

\def\supp{\operatorname{supp}}

\def\stratmax{|S|}
\def\stratmin{|T|}

\def\MPGstrong{\operatorname{MPGI*}}
\def\MPGinteger{\operatorname{MPGI}}

\def\leqlex{\leq_{\makebox{lex}}}

\newcommand{\mymax}{\mathop{\text{\Large$\vee$}}}%
\newcommand{\mymin}{\mathop{\text{\Large$\wedge$}}}%

\begin{document} 
 
\title{Tropical linear-fractional programming and parametric mean payoff games} 
 
\author[Gaubert]{St\'ephane Gaubert}

\address{INRIA and Centre de Math\'ematiques Appliqu\'ees, 
\'Ecole Polytechnique. Postal address: CMAP, \'Ecole Polytechnique, 
91128 Palaiseau C\'edex, France} 
\email[Gaubert]{Stephane.Gaubert@inria.fr}
 
\author{Ricardo D. Katz} 
\address{CONICET. Postal address: Instituto de Matem\'atica ``Beppo Levi'', 
Universidad Nacional de Rosario, Av. Pellegrini 250, 2000 Rosario, Argentina} 
\email{rkatz@fceia.unr.edu.ar}

\author[Sergeev]{Serge\u{\i} Sergeev}
\address{INRIA and Centre de Math\'ematiques Appliqu\'ees, 
\'Ecole Polytechnique. Postal address: CMAP, \'Ecole Polytechnique, 
91128 Palaiseau C\'edex, France} 
\email{sergeev@cmap.polytechnique.fr}

\thanks{The first author was partially supported by the
Arpege programme of the French National Agency of Research (ANR), 
project ``ASOPT'', number ANR-08-SEGI-005 and by the Digiteo project
DIM08 ``PASO'' number 3389. The third author was supported by the EPSRC grant RRAH12809 and the RFBR-CNRF grant 11-01-93106. This work was initiated when this author was with the School of Mathematics at the University of Birmingham.} 
 
\begin{abstract} 
Tropical polyhedra have been recently used to represent 
disjunctive invariants in static analysis.
To handle larger instances, tropical analogues of
classical linear programming results need to be developed.
This motivation leads us to study the tropical analogue
of the 
classical linear-fractional programming problem. 
We construct an associated parametric mean payoff game
problem, and show that the optimality of a given point, or the
unboundedness of the problem, can be certified by exhibiting a
strategy for one of the players having certain infinitesimal properties
(involving the value of the game and its derivative) that we
characterize combinatorially. 
We use this idea to design a
Newton-like algorithm to solve tropical linear-fractional programming problems,
by reduction to a sequence of auxiliary mean payoff game problems. 
\end{abstract} 
 
 
\maketitle
 
\section{Introduction}\label{s:introduction}

\subsection{Motivation from static analysis}
Tropical algebra is the structure in which the set of real numbers,
completed with $-\infty$, is equipped with the
``additive'' law $\ltr a+b\rtr:=a\vee b=\max(a,b)$ 
and the ``multiplicative''
law $\ltr ab\rtr:=a+b$.
The max-plus or tropical analogues of convex sets have been studied
by a number of authors~\cite{Zim-77,CG:79,maxplus97,LMS-01,CGQ-04,DS-04,BriecHorvath04,joswig04}, under various names (idempotent spaces,
semimodules, $\mathbb{B}$-convexity, extremal convexity), with
different degrees of generality, and various motivations. 

In the recent work~\cite{AGG08}, Allamigeon, Gaubert and Goubault have
used tropical polyhedra to compute disjunctive invariants in static
analysis.
A general (affine) tropical polyhedron can be represented as
\begin{align}\label{e-tpol}
P:=\{ x\in (\R\cup\{-\infty\})^n\mid \big(\mymax_{j\in[n]} (a_{ij}+x_j)\big)\vee c_i\leq \big(\mymax_{j\in[n]} (b_{ij}+x_j)\big)\vee d_i  ,  \forall i\in[m] \}  .
\end{align}
Here, we use the notation $[n]:=\{1,\ldots ,n\}$,
and the parameters $a_{ij}$, $b_{ij}$, $c_i$ and $d_i$
are given, with values in $\R\cup\{-\infty\}$. The analogy with
classical polyhedra becomes clearer with the tropical notation,
which allows us to write the constraints as $\ltr Ax+c\leq Bx+d\rtr$,
to be compared with classical systems of linear inequalities,
$Ax\leq d$ (in the tropical setting, we need to consider affine functions 
on both sides of the inequality due to the absence of opposite law
for addition). The previous representation of $P$ is the analogue
of the {\em external} representation of polyhedra, as the intersection
of half-spaces. As in the classical case, tropical polyhedra
have a dual (internal) representation, which involves extreme points and
extreme rays. The tropical analogue of
Motzkin double description method allows one to pass from one
representation to the other~\cite{AGG10}.

Disjunctive invariants arise naturally when analyzing sorting algorithms
or in the verification of string manipulation programs. 
The well known {\tt memcpy} function of C is discussed
in~\cite{AGG08} as a simple illustration: when copying
the first {\tt n} characters of a string buffer {\tt src} to a string buffer
{\tt dst}, the length {\tt len\_dst} of the latter buffer may
differ from the length {\tt len\_src} of the former, for if {\tt n}
is smaller than {\tt len\_src}, the null terminal character
of the buffer {\tt src} is not copied. However, the relation
$\min(\mathtt{len\_src},\mathtt{n}) = \min(\mathtt{len\_dst},\mathtt{n})$
is valid. This can be expressed geometrically by saying that
the vector  
$(-\mathtt{len\_src},-\mathtt{len\_dst})$
belongs to a tropical polyhedron. 
Several examples of programs of a disjunctive nature, which
are analyzed by means of tropical polyhedra, can be found in~\cite{AGG08,AllamigeonThesis}, in which the tropical analogue of the classical polyhedra-based
abstract interpretation method of~\cite{CousotHalbwachs78-POPL} has been developed.

The comparative interest of tropical polyhedra is illustrated in
Figure~\ref{f:abstractions}, which gives a simple fragment of code
in which the tropical invariant is tighter. Note that there is still
an over-approximation in the tropical case, because the transfer function
considered here is discontinuous
 (tropical polyhedra share
with classical polyhedra the property of being connected, and therefore
cannot represent exactly such discontinuities). Such tropical
invariants can be obtained automatically via the methods of~\cite{AGG08,AllamigeonThesis}, which rely on the tropical analogue of the double description algorithm~\cite{AGG10}, allowing one to obtain the vertices of a tropical polyhedron from a
family of defining inequalities, and vice versa. 

As in the case  of classical polyhedra,   
the scalability of the approach is inherently limited by the 
exponential blow up of the size of representations of 
polyhedra, since the number of vertices or of defining inequalities
can be exponential in the size of the input data~\cite{AGK-09,AGK-10}.

The complexity of earlier polyhedral approaches
led Sankaranarayanan, Colon, Sipma and Manna
to introduce the method of {\em templates}~\cite{Sriram1,Sriram2}. 
In a nutshell, a template
consists of a finite set  $\mathcal{T}=\{\objective_1, \ldots,\objective_m\}$
of linear forms on $\R^n$. 
The latter define a parametric family of polyhedra
\[
P_\alpha=\{x\in \R^n\mid \objective_k(x) \leq \alpha_k\enspace ,\enspace k\in [m]\} 
\]
with precisely
$m$ degrees of freedom $\alpha_1,\ldots,\alpha_m\in \R\cup\{+\infty\}$. 
The classical domains of boxes or the domain of {\em zones} 
(potential constraints)~\cite{PhDMine} are recovered by 
incorporating in the template the linear forms 
$\objective(x)=\pm x_i$ or $\objective(x)=x_i-x_j$, 
respectively. 
Fixing the template, or changing it
dynamically while keeping $m$ bounded, avoids the exponential blow up.

The method of~\cite{Sriram2} relies critically on linear programming, 
which allows one to evaluate quickly the fixed point functional 
of abstract interpretation. However, 
the precision of the invariants remains limited
by the linear nature of templates, 
and it is natural to ask whether the machinery 
of templates carries over to the non-linear case. 

\begin{figure}
\begin{center}
\begin{tabular}{ccc}
\begin{minipage}{0.20\textwidth}
\begin{lstlisting}
if x > y then 
  z:=x;
else
  z:=y+1;
\end{lstlisting}%
\end{minipage}&
\input{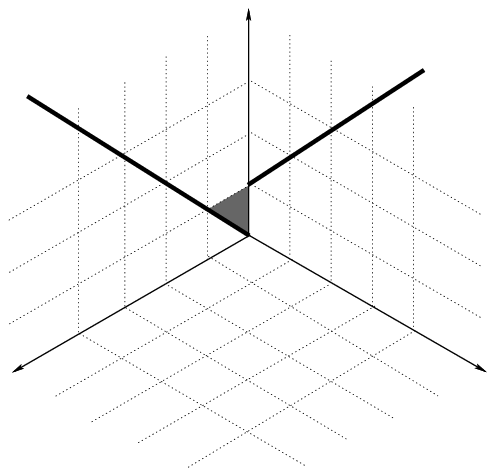} &
\input{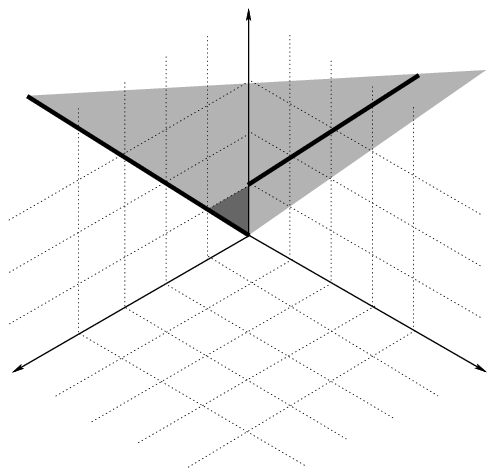}\\[-2em]
& $\begin{array}{l}
x\leq z\\
y\leq \max(x,z-1)\\
z\leq \max(x,y+1)
\end{array}$
& $\begin{array}{l}
x\leq z\\
y\leq z 
\end{array}$
\end{tabular}
\end{center}
\caption{The tropical polyhedral abstraction of a transfer function, following~\cite{AGG08}, versus the classical 
polyhedral abstraction. The optimal invariant is represented by the union of two half-lines. The tropical over-approximation of this behavior (union of the triangle in dark gray and of the half-lines, left) is more accurate than the classical polyhedral one (cone in light gray, right).
}
\label{f:abstractions}
\end{figure}

The formalism of templates has been extended to the non-linear
case by Adj\'e, Gaubert and Goubault~\cite{AssaleGG},
who considered specially quadratic templates, 
the linear programming methods of~\cite{Sriram2} being replaced by semidefinite programming, 
thanks to Shor's relaxation. More generally,
every tractable subclass of optimization problems yields
a tractable template. 

In order to compute disjunctive invariants,
Allamigeon, Gaubert and Goubault suggested
to develop a generalization of the template method to the case of tropical
polyhedra. As a preliminary step, the relevant results of linear programming
must be tropicalized: this is the object of the present
paper.

In particular, comparing the expressions of $P$ and $P_\alpha$,
we see that the classical linear forms %
must now be replaced by the differences of tropical affine forms
\begin{align}\label{theobjective}
\objective(x)= 
\big(\mymax_{j\in [n]} (p_j+x_j)\vee r\big)- \big(\mymax_{j\in [n]} (q_j+x_j)
\vee s \big) 
\end{align}
where $p_j$, $r$, $q_j$ and $s$ are given parameters with values in $\R\cup\{-\infty\}$.

\subsection{The problem}

In this paper, 
we study the following {\em tropical linear-fractional programming problem}:
\begin{align}\label{pblp}
\begin{split}
&\text{minimize }\quad  \objective(x) \\
&\text{subject to:}\quad   x\in P 
\end{split} 
\end{align}
where $P$ is given by~\eqref{e-tpol} and $\objective$ 
is given by~\eqref{theobjective}.
This is the tropical analogue of the classical linear-fractional programming problem
\[
\begin{split}
&\text{minimize } \quad (px+r)/(qx+s) \\
&\text{subject to:}\quad Ax+c\leq Bx+d \; ,\;
x\geq 0 \; , \; x\in \R^n 
\end{split}
\]
where $p$, $q$ are nonnegative vectors, $r$, $s$ are nonnegative scalars,
and $A$, $B$, $c$, $d$ are matrices and vectors. The constraint $\ltr x\geq 0\rtr$ is
implicit in~\eqref{pblp}, since any number is ``nonnegative'' (i.e.\ $\geq -\infty$) in the tropical world.

Problem~\eqref{pblp} includes as special cases 
\begin{align}\label{special1}
\objective(x)=\mymax_{j\in [n]} (p_j+x_j)
\qquad\text{and}\qquad 
\objective(x)=- \mymax_{j\in [n]} (q_j+x_j) 
\end{align}
(take $q_j\equiv-\infty$, $r=-\infty$ and $s=0$, or $p_j\equiv-\infty$, $s=-\infty$ and $r=0$). According to the terminology of~\cite{BA-08},
the latter may be thought
of as the tropical analogues of the linear programming problem. 
However, optimizing the more general fractional objective function~\eqref{theobjective} appears to be needed in a number of basic applications. 
In particular, in static analysis,
we need typically to compute the
tightest inequality of the form $x_i\leq K+x_j$ satisfied by the elements
of $P$. This fits in the general form~\eqref{theobjective}, but not in
the special cases~\eqref{special1}. 

\subsection{Contribution}
A basic question in linear programming is to certify the optimality
of a given point. This is classically done by exhibiting a feasible
solution of the dual problem (i.e.\ a vector of Lagrange multipliers)
with the same value. There is no such a simple result in the tropical
setting, because as remarked in~\cite{GK-09}, 
there are (tropically linear) inequalities which can be logically deduced from 
some finite system of (tropically linear) inequalities but which cannot be 
obtained by taking (tropical) linear combinations of the inequalities of 
the system. In other words, the usual statement of Farkas
lemma is not valid in the tropical setting. However, 
recently Allamigeon, Gaubert and Katz~\cite{AGK-10} established 
a tropical analogue of Farkas lemma, 
building on~\cite{AGG-10}, in which Lagrange
multipliers are replaced by strategies of an associated mean payoff game.
We use the same idea here, and show in Subsection~\ref{ss:certificates} (Theorem~\ref{1st-cert} below) that the
optimality of a solution can be (concisely) certified by exhibiting a strategy
of a game, having certain combinatorial properties. 
Similarly, whether the value
of the tropical linear-fractional programming problem is unbounded can
also be certified in terms of strategies (Theorem~\ref{2nd-cert}).

The second ingredient is to think of the tropical linear-fractional programming
problem as a parametric mean payoff game problem. Then,
the tropical linear-fractional programming problem reduces to the computation of
the minimal parameter for which the value of the game is nonnegative. 
As a function of the parameter, this value is  
piecewise-linear and $1$-Lipschitz, 
see Subsection~\ref{ss:specf} for a more detailed description. 

The main contribution of this paper is a Newton-like method, 
where at each iteration we select a strategy playing the role of derivative 
(whose existence is implied by the fact that 
the current feasible point is not optimal).
This defines a {\em one player} parametric game problem, 
and we show that the smallest value of the parameter making the
value zero can be computed in polynomial time for this subgame, 
by means of shortest path algorithms as described in Subsection~\ref{ss:kleene}.
The master algorithm (Algorithm~\ref{a:pos-newton}) requires 
solving at each iteration an auxiliary mean payoff game, 
see Subsections~\ref{ss:left-optim} and~\ref{ss:germs}. 
Mean payoff games can be solved either by value iteration,
which is pseudo-polynomial, or by policy iteration, for which
exponential time instances have been recently constructed
in~\cite{Friedmann-AnExponentialLowerB}, 
although the algorithm is fast on typical examples. 
The number of Newton type iterations of the Newton-like algorithm
has a trivial exponential bound (the number of strategies), and we
show that the algorithm is pseudo-polynomial, see Theorem~\ref{posnewt-comp}.
Although this algorithm seems to behave well on typical examples, 
see in particular Subsection~\ref{SectionExampleMin},
some further work would be needed to assess its worst case
complexity (its behavior is likely to be similar to the one
of policy iteration, see Subsection~\ref{ss:NumExp} 
for a preliminary account of numerical experiments). 

\subsection{Related work}
Butkovi\v{c} and Aminu~\cite{BA-08} studied the special 
cases~\eqref{special1}, see also~\cite[Chapter 10]{But:10}. At each iteration, 
they solve a feasibility problem (whether a tropical polyhedron is non-empty), 
which is equivalent to checking whether 
a mean payoff game is winning for one of the players. 
However, their algorithm does not involve a Newton-like iteration,
but rather a dichotomy argument. 
In~\cite{BA-08} the number of calls to a mean payoff oracle, 
whose implementation relies on the alternating method of~\cite{CGB-03}, 
depends on the size of the integers in the input, 
whereas the number of calls in the present algorithm 
can be bounded independently of these, just in terms of strategies. 

A different approach to tropical linear programming was developed 
previously in the works of U. Zimmermann~\cite{Zim:81} 
and K. Zimmermann~\cite{Zim-05}. This approach,
which is based on residuation theory, 
and works over more general idempotent semirings,
identifies important special cases in which the solution
can be obtained explicitly (and often, in linear time). 
However,  it cannot be applied to our more general formulation, 
in which there is little hope to find similar explicit solutions. 

Newton methods for finding the least fixed point of nonlinear 
functions are also closely related to this paper. 
Such methods were developed in~\cite{Policy1,ESOP07,Seidl2} to solve monotone fixed point
problems arising in abstract interpretation.
Esparza et al.~\cite{esparza:approximative} develop such methods for monotone systems of 
min-max-polynomial equations. 
They seek for the least fixed point of a function whose
components are max-polynomials or min-polynomials, 
showing that their Newton methods have linear convergence at least. 
The class of functions considered there is considerably more
general than the tropical linear forms appearing here,  but the
fixed point problems considered in~\cite{esparza:approximative} appear
to be of a different nature. It would be interesting, however,
to connect the two approaches. 

\subsection{Further motivation} 

Tropical polyhedra have been used in~\cite{katz05} to determine
invariants of discrete event systems. 
Systems of constraints equivalent to the ones defining tropical polyhedra
have also appeared
in the analysis of delays in digital circuits, 
and in the study of scheduling problems with both 
``and'' and ``or'' constraints~\cite{mohring}. 
Such systems have been studied by Bezem, Nieuwenhuis, 
and Rodr\'\i guez-Carbonell~\cite{bezem2,bezemjournal}, 
under the name of the ``max-atom problem''. 
The latter is motivated by SAT Modulo theory solving, 
since conjunctions of max-atoms determine a 
remarkable fragment of linear arithmetic. 
Tropical polyhedra also turn out to be
interesting mathematical objects in their own right~\cite{DS-04,joswig04}.
A final motivation arises from mean payoff games, the complexity
of which is a well known open problem: a series of works
show that a number of problems which can 
be expressed in terms of tropical polyhedra
are polynomial time equivalent to mean payoff games 
problems~\cite{mohring,DG-06,AGG-10,bezemjournal,AGK-10}. 

The results of the present paper rely on~\cite{AGG-10,AGK-10}. 

\section{Preliminaries}

In this section, we recall some definitions and results
needed to describe the present tropical linear-fractional
programming algorithm.

We start by introducing deterministic mean payoff games played
on a finite bipartite digraph, see Subsection~\ref{mpg-start}. 
We then summarize some elements of the operator approach to such games,
in including a theorem of~\cite{Koh-80}, 
Theorem~\ref{t:inv-half} below, implying that their dynamic programming 
operators $f$ have invariant half-lines $(\chi, v)$. 
These invariant half-lines determine the ultimate growth of the orbits
$f^k(x)/k$, known as the cycle-time vector $\chi(f)$ of $f$. 
They also determine a pair of optimal strategies, 
and the winning nodes of the players
($\{i\mid \chi_i(f)\geq 0\}$ and $\{i\mid \chi_i(f)<0\}$) in the associated game. 

Then, we recall in Subsection~\ref{trop-mpg} the correspondence
between tropical polyhedra and mean payoff games, along the lines
of~\cite{AGG-10}: 
Theorem~\ref{chi-axbx} shows that $i$ is a winning node $(\chi_i(f)\geq 0)$
if, and only if, the associated tropical two-sided system of inequalities 
has a solution $x$ whose $i$th coordinate is finite $(x_i\neq -\infty)$.
We also recall the max-plus and min-plus representations of min-max 
functions, together with the combinatorial characterization,
in terms of maximal or minimal cycle means, of the 
cycle time vector for one-player games. This will be used
to construct optimality and unboundedness certificates
in Subsection~\ref{ss:certificates}. 

\subsection{Mean payoff games and min-max functions}\label{mpg-start}

Consider a two-player deterministic game where the players, 
called ``Max'' and ``Min'', make alternate moves 
of a pawn on a weighted bipartite digraph $\Bipdig$. 
The set of nodes of $\Bipdig$ is the disjoint
union of the nodes $[m]$ where Max is active, 
and the nodes $[n]$ where Min is active. 
When the pawn is in node $k\in [m]$ of Max, 
he must choose an arc in $\Bipdig$
connecting node $k$ with some node $l\in [n]$ of Min, 
and while moving the pawn along this arc, 
he receives the weight $b_{kl}$ of the selected arc as payment from Min. 
When the pawn is in node $j\in [n]$ of Min, 
she must choose an arc in $\Bipdig$
connecting node $j$ with some node $i\in [m]$ of Max, 
and pays $-a_{ij}$ to Max, where $-a_{ij}$ is the weight of the selected arc. 
We assume that $b_{kl},a_{ij}\in\R$. Moreover, certain moves may be prohibited,
meaning that the corresponding arcs are not present in $\Bipdig$.
Then, we set $b_{kl}=-\infty$ and $a_{ij}=-\infty$. Thus, 
the whole game is equivalently defined by two $m\times n$ matrices
$A=(a_{ij})$ and $B=(b_{kl})$ with entries in $\R\cup\{-\infty\}$.
We make the following assumptions, 
which assure that both players have at least one move allowed in each node.

{\bf Assumption 1.} For all $k\in[m]$ there exists $l\in [n]$
such that $b_{kl}\neq -\infty$.

{\bf Assumption 2.} For all $j\in[n]$ there exists $i\in[m]$
such that $a_{ij}\neq -\infty$.

A {\em general strategy} for a player (Max or Min) is a function 
that for every finite history of a play ending at some node  
selects a successor of this node (i.e., a move of the player). 
A {\em positional strategy} for a player is a function that selects a 
unique successor of every node independently of the 
history of the play.

A strategy for player Max will be usually denoted by $\sigma$ and 
a strategy for player Min by $\tau$. Thus, 
a positional strategy for player Max is a function 
$\sigma\colon [m] \mapsto [n]$ 
such that $b_{i\sigma(i)}$ is finite for all $i\in[m]$, 
and a positional strategy for player Min is a function 
$\tau\colon [n] \mapsto [m]$
such that $a_{\tau(j)j}$ is finite for all $j\in [n]$. 

When player Max reveals his positional strategy $\sigma$, 
the play proceeds within the digraph $\Bipdig^{\sigma}$ 
where at each node $i$ of Max all but one arc $(i,\sigma(i))$ are removed. 
When player Min reveals her positional strategy $\tau$, 
the play proceeds within the digraph $\Bipdig^{\tau}$ 
where at each node $j$ of Min all but one arc $(j,\tau(j))$ are removed.

When player Max reveals his positional strategy $\sigma$
and player Min her positional strategy $\tau$, 
the play proceeds within the digraph $\Bipdig^{\sigma,\tau}$ 
where each node has a unique outgoing arc $(i,\sigma(i))$ or $(j,\tau(j))$. 
Thus, $\Bipdig^{\sigma,\tau}$ is a ``sunflower'' digraph, i.e., 
such that each node has a unique path to a unique cycle.  

The {\em infinite horizon version} of this mean payoff game  
starts at a node $j$ of Min\footnote{The game can start also 
at a node $i$ of Max, the requirement to be started by Min is 
for better consistency with min-max functions and two-sided tropical systems, 
see~\eqref{minmax-dynamic} and~\eqref{ineqs-equiv}.} 
and proceeds according to the strategies (not necessarily positional) 
of the players, who are interested in the value of 
average payment. More precisely, player Min wants to minimize 
\[
\Phi_{A,B}^{\sup}(j,\tau,\sigma):=\limsup_{k\to\infty}  \; (\sum_{t=1}^k -a_{i_t j_{t-1}}+b_{i_t j_t} )/k
\enspace, \quad j_0=j \enspace , 
\]
while player Max wants to maximize
\[  
\Phi_{A,B}^{\inf}(j,\tau,\sigma):=\liminf_{k\to\infty}  \; (\sum_{t=1}^k -a_{i_t j_{t-1}}+b_{i_t j_t} )/k
\enspace, \quad j_0=j \enspace ,
\]
where $j_1\in [n]$, $i_1\in [m]$, $j_2\in [n]$, $i_2\in[m]$, $\ldots$ 
is the infinite sequence of positions of the pawn resulting from 
the selected strategies $\tau $ and $\sigma $ of the players.
The next theorem shows that this game has a value. 
An analogue of this theorem concerning stochastic 
games was obtained by \cite{liggettlippman}.

\begin{theorem}[\cite{EM-79,GKK-88}]\label{value} 
For the mean payoff game whose payments are given by the matrices 
$A,B\in (\R\cup\{-\infty\})^{m\times n}$, 
where $A,B$ satisfy Assumptions 1 and 2, 
there exists a vector $\chi\in\R^n$ and a pair of
positional strategies $\sigma^*$ and $\tau^*$ such that 
\begin{enumerate}[(i)]
\item $\Phi_{A,B}^{\sup}(j,\tau^*,\sigma)\leq \chi_j$ for
all (not necessarily positional) strategies $\sigma$,
\item $\Phi_{A,B}^{\inf}(j,\tau,\sigma^*)\geq \chi_j$
for all (not necessarily positional) strategies $\tau$, 
\end{enumerate}
for all nodes $j$ of Min. 
\end{theorem}

In other words, player Max has a positional strategy $\sigma^*$ which secures 
a mean profit of at least $\chi_j$ whatever is the strategy of player Min, 
and player Min has a positional strategy $\tau^*$ which secures a mean 
loss of no more than $\chi_j$ whatever player Max does.

A {\em finite duration version} of the mean payoff game considered
above can be also formulated. Again, 
it starts at a certain  node $j$ of Min and 
proceeds according to the strategies of the players  
(not necessarily positional), but stops immediately when a cycle 
$j_0\in [m] \rightarrow i_1\in [n] \rightarrow j_1\in [m] 
\rightarrow \cdots \rightarrow i_k\in [n] \rightarrow j_k=j_0$ 
is formed. Then, the outcome of the game is the mean weight per turn 
(so the length of a cycle may be seen as the number of nodes 
of Max or Min it contains) of that cycle:
\begin{equation}\label{PhiAB}
\Phi_{A,B}(j,\tau,\sigma):=(\sum_{t=1}^k -a_{i_t j_{t-1}}+b_{i_t j_t})/k \; .
\end{equation}
The ambition of Max is to maximize 
$\Phi_{A,B}(j,\tau,\sigma)$ while Min is seeking to minimize it. 

It was shown in~\cite[Theorem 2]{EM-79} that this finite duration version of 
the game has the same value as the infinite horizon version described above, 
and that there are positional optimal strategies which secure this value 
for both versions of the game. 
It follows by standard arguments that $\chi_j$ is determined uniquely by
\begin{equation}\label{val-exist}
\chi_j=\Phi_{A,B}(j,\tau^*,\sigma^*)=
\min_{\tau}\max_{\sigma} \Phi_{A,B}(j,\tau,\sigma)=
\max_{\sigma}\min_{\tau} \Phi_{A,B}(j,\tau,\sigma) \; ,
\end{equation}
where $\tau$ and $\sigma$ range over the sets of all strategies 
(not necessarily positional) for players Min and Max, respectively. 

The dynamic programming operator $f: \R^n\mapsto\R^n$ associated with 
the infinite horizon mean payoff game is defined by: 
\begin{equation}\label{minmax-dynamic}
f_j(x)=\mymin\limits_{k\in[m]}(-a_{kj}+\mymax\limits_{l\in[n]}(b_{kl}+x_l)) \enspace .
\end{equation}
This function, combining min-plus and max-plus linearity (see below in Subsection~\ref{trop-mpg}),
is known as a {\em min-max function}~\cite{CGG-99}.
Min-max functions are
isotone ($x\leq y\Rightarrow f(x)\leq f(y)$) and additively homogeneous
($f(\lambda+x)=\lambda+f(x)$). Hence, they are
nonexpansive in the sup-norm. Moreover, they are piecewise
affine ($\R^n$ can be covered by a finite number of
polyhedra on which $f$ is affine). We are interested
in the following limit ({\em cycle-time vector}):
\begin{equation}
\label{chi-limit}
\chi(f)=\lim_{k\to\infty} f^k(x)/k \enspace .
\end{equation}
The $j$th entry of the
vector $\chi(f)$ can be interpreted as the limit of the mean value of the
game per turn, as the horizon $k$ tends to infinity,
when the starting node is $j$. 
The existence of $\chi(f)$ follows from a theorem of Kohlberg. 

\begin{theorem}[\cite{Koh-80}]\label{t:inv-half}
Let $f:\R^n\mapsto\R^n$ be a nonexpansive
and piecewise affine function. Then, there exist $v\in\R^n$
and $\chi\in\R^n$ such that
\[
f(v+t\chi)=v+(t+1)\chi\enspace,\quad \forall t\geq T \enspace ,
\]
where $T$ is a large enough real number.
\end{theorem}

The function $t\mapsto v+t\chi$ is known as an {\em invariant half-line}.
Using the nonexpansiveness of $f$, one deduces that the limit~\eqref{chi-limit}
exists, is the same for all $x\in\R^n$ and is equal to the growth
rate $\chi$ of any invariant half-line.

Given fixed positional strategies $\tau$ and $\sigma$ for players Min and Max, 
respectively, we can consider the dynamic operators corresponding to the 
partial digraphs $\Bipdig^{\tau}$ and $\Bipdig^{\sigma}$.
These operators are max-only and min-only functions:
\begin{equation}\label{mm-only}
\begin{split}
f_j^{\tau}(x)&=-a_{\tau(j)j}+\mymax\limits_{l\in [n]} ( b_{\tau(j)l}+x_l ) \enspace ,\\
f_j^{\sigma}(x)&=\mymin\limits_{i\in [m]}(-a_{ij}+b_{i\sigma(i)}+x_{\sigma(i)})\enspace .
\end{split}
\end{equation}
They are the main subject of tropical linear algebra, 
see Subsection~\ref{trop-mpg}, where in particular we recall 
how their cycle-time vectors can be computed. 
Theorem~\ref{e:chi-duality-games} below relates these 
cycle-time vectors with the 
cycle time vector of the min-max function~\eqref{minmax-dynamic}. 

The following result can be derived as a standard corollary of Kohlberg's 
theorem. Indeed, we define a positional strategy $\tau$ of Min and
a positional strategy $\sigma$ of Max by the condition that
$f(v+t\chi)=f^\sigma(v+t\chi)=f^\tau(v+t\chi)$ for $t$ large
enough, where $t\mapsto v+t\chi$ is an invariant half-line.
These strategies are easily seen to be optimal for the mean payoff game.

\begin{theorem}[Coro.\ of~\cite{Koh-80}]\label{value-cycletime} 
For $f(x)$ given by~\eqref{minmax-dynamic},
the $j$th coordinate of $\chi(f)$ is the value of the mean payoff
game which starts at node $j$ of Min.
\end{theorem}

In what follows, we shall use the following form of the value 
existence result~\eqref{val-exist}, 
which was proved in~\cite{gg0} as a corollary of the termination of
the policy iteration algorithm of~\cite{gg0,CGG-99},
see~\cite{DG-06} for a more recent presentation.
Alternatively, it can be quickly derived from~\cite{Koh-80}
(the derivation can be found in~\cite[Theorem~2.13]{AGG-10}).
This result has been known as the ``duality theorem''
in the discrete event systems literature.

\begin{theorem}[Coro.\ of~\cite{Koh-80},\ \cite{gg0}]
\label{chi-duality}
Let $A,B\in (\R\cup\{-\infty\})^{m\times n}$ satisfy
Assumptions 1 and 2, and let $S$ and $T$ be the sets of all
positional strategies for players Max and Min, respectively. Then,
\begin{equation}\label{e:chi-duality-games}
\max\limits_{\sigma\in S} \chi(f^{\sigma})=\chi(f)=
\min\limits_{\tau\in T} \chi(f^{\tau}) \enspace .
\end{equation}
\end{theorem}
This characterization of $\chi(f)$
should be compared with~\eqref{val-exist}. The latter shows that the
infinite horizon version of the game with limsup/liminf payoff has a value, whereas~\eqref{e:chi-duality-games} concerns the limit of the value of the
finite horizon version. Thus, in loose terms, the ``limit'' and ``value'' operations commute.
\subsection{Tropical linear systems and mean payoff games}\label{trop-mpg}

Max-only and min-only functions of the form~\eqref{mm-only} belong to
tropical linear algebra. Max-only functions are linear in the {\em
max-plus semiring} $\Rmax$, which is the set $\R\cup\{-\infty\}$ 
equipped with the operations of ``addition'' 
$\ltr a+b\rtr:=a\vee b$   
and ``multiplication'' $\ltr ab\rtr:=a+b$. For min-only functions, we use
the {\em min-plus semiring} $\Rmin$, i.e.\ the set $\R\cup\{+\infty\}$ 
equipped with the operations of ``addition'' 
$\ltr a+b\rtr:=a\wedge b$ 
and the same ``additive'' multiplication.
The setting in which both structures are considered
simultaneously has been called minimax algebra 
in~\cite{CG:79}. Then, we need
to allow the scalars to belong to the enlarged set 
$\RRbar:=\R\cup\{-\infty\}\cup \{+\infty\}$.
Note that in $\RRbar$,
$(-\infty)+(+\infty)=-\infty$ if the max-plus convention
is considered and $(-\infty)+(+\infty)=+\infty$ if the min-plus convention
is considered. 

The tropical operations are extended to matrices and vectors in the usual way. 
In particular, for any matrix $E=(e_{ij})$ and any vector $x$ of 
compatible dimensions:
\begin{equation}\label{DefMult}
\ltr (Ex)_{i}\rtr=\mysup_j e_{ij}+x_j\quad \text{(max-plus)}\; , \quad
\ltr (Ex)_{i}\rtr=\myinf_j e_{ij}+x_j\quad \text{(min-plus)}\; .
\end{equation} 
Max-plus and min-plus linear functions are mutually adjoint, or
{\em residuated}. Recall that for a
max-plus linear function from $\RRbar^n$ to $\RRbar^m$, 
given by $E\in \Rmax^{m\times n}$, the 
{\em residuated operator} $E^{\diez}$ from $\RRbar^m$ to $\RRbar^n$ is
defined by
\begin{equation}
\label{adiez}
(E^{\diez} y)_j:=\myinf_{i\in [m]} (-e_{ij}+y_i)\enspace ,
\end{equation}
with the convention $(-\infty)+(+\infty)=+\infty$.  
Note that this residuated operator, also known as
{\em Cuninghame-Green inverse}, 
is given by the multiplication of $-E^T$ by $y$ with the min-plus 
operations (here $E^T$ denotes the transposed of $E$), and that it 
sends $\RR^m$ to $\RR^n$ whenever $E$ does not have 
columns identically equal to $-\infty$. 

In what follows, concatenations such as $Ex$ should be understood as the 
multiplication of $E$ by $x$ with the max-plus operations, 
and concatenations such as $E^{\diez}y$ should be understood 
as the multiplication of $-E^T$ by $y$ with the min-plus operations 
(and the corresponding conventions for $(-\infty)+(+\infty)$). 

The term ``residuated'' refers to the property
\begin{equation}\label{res-prop}
Ex\leq y\Leftrightarrow   
x\leq E^{\diez} y\enspace ,
\end{equation}
where $\leq$ is the partial order on $\RR^m$ or $\RR^n$, 
which can be deduced from
\[  
e_{ij} + x_j \leq y_i \; \forall i, j
\Leftrightarrow 
x_j \leq -e_{ij} +y_i  \; \forall i, j \enspace .
\]
As a consequence, the residuated operator is crucial for max-plus 
two-sided systems of inequalities, because~\eqref{res-prop} implies: 
\begin{equation}\label{ineqs-equiv}
Ax\leq Bx\Leftrightarrow x\leq A^{\sharp}Bx \enspace .
\end{equation}
Writing the last inequality explicitly, we have
\begin{equation}\label{ineqs-expl}
x_j\leq\myinf_{k\in [m]}(-a_{kj}+\mysup_{l\in [n]}(b_{kl}+x_l))
\enspace , \quad\forall j\in [n]\enspace .
\end{equation}
Thus, we obtain the same min-max function as in~\eqref{minmax-dynamic}.

Moreover, positional strategies
$\sigma\colon [m] \mapsto [n]$ and
$\tau\colon [n]\mapsto [m]$ correspond
to affine functions $B^{\sigma}$ and $A_{\tau}$ defined by
\begin{equation}\label{ataubsigma}
(A_{\tau})_{ij}=
\begin{cases}
a_{ij} \; &\makebox{ if } i=\tau(j) ,\\
-\infty \; &\makebox { otherwise} ,
\end{cases}\quad
(B^{\sigma})_{ij}=
\begin{cases}
b_{ij}\; &\makebox{ if }  j=\sigma(i) ,\\
-\infty \; &\makebox { otherwise} .
\end{cases}
\end{equation}

Recasting~\eqref{e:chi-duality-games} in max(min)-plus algebra,
we obtain
\begin{equation}\label{e:chi-duality}
\max\limits_{\sigma\in S}\chi(A^{\sharp}B^{\sigma}) = \chi(A^{\sharp}B) = 
\min\limits_{\tau\in T}\chi(A^{\sharp}_{\tau}B) \enspace .
\end{equation}

The following result, obtained by Akian, Gaubert and Guterman,
relates solutions of $Ax\leq Bx$ and nonnegative coordinates
of $\chi(A^{\sharp}B)$. These coordinates correspond
to {\em winning nodes} of the mean payoff game: 
if the game starts in these nodes, 
then Max can ensure nonnegative profit with any strategy of Min. 

\begin{theorem}[{\cite[Th.~3.2]{AGG-10}}]\label{chi-axbx} 
Let $A,B\in\Rmax^{m\times n}$ satisfy Assumptions~1 and~2.
Then, $\chi_i(A^{\sharp}B)\geq 0$ if and only if
there exists $x\in\Rmax^n$ such that $Ax\leq Bx$
and $x_i\neq -\infty$.
\end{theorem}

This is derived in~\cite{AGG-10} from
Kohlberg's theorem. The vector $x$ is constructed
by taking an invariant half-line, $t\mapsto v+t\chi$, 
setting $x_i=v_i+t\chi_i$ for $t$ large enough if $\chi_i\geq 0$, 
and $x_i=-\infty$ otherwise. 

Theorem~\ref{chi-axbx} shows that to decide whether $Ax\leq Bx$ can
be satisfied by a vector $x$ such that $x_i\neq -\infty$, we can
exploit a {\em mean payoff oracle}, which decides whether
$i$ is a winning node of the associated mean payoff game and gives a
winning strategy for player Max. This oracle can be implemented
either by using the value iteration method, which is pseudo-polynomial~\cite{ZP-96}, by the approach of Puri (solving an associated discounted game for
a discount factor close enough to $1$ by policy iteration~\cite{puri}),
by using the policy iteration algorithm for mean payoff games 
of \cite{CGG-99,gg0,DG-06}, or the one of~\cite{bjorklund}.

In tropical linear algebra, 
there is no obvious subtraction. However,
for any $E\in\RRbar^{n\times n}$ we can define the {\em Kleene star}
\begin{equation}\label{kls-def}
E^*:=\ltr(I-E)^{-1}\rtr=I\vee E\vee E^2\vee\cdots\quad \text{(max-plus)}\; ,
\end{equation}
and analogously with $\wedge$ in the min-plus case. 
In~\eqref{kls-def}, $I$ is the max-plus identity matrix with $0$ entries 
on the main diagonal and $-\infty$ off the diagonal, and 
the powers are understood in the tropical (max-plus) sense. 
Due to the order completeness of $\RRbar$, 
the series in~\eqref{kls-def} is well-defined for all matrices.
Note that in $\RRbar^{n\times n}$, $X=E^*$ is a solution of 
the matrix Bellman equation $X=E X\vee I$. Similarly, 
$x=E^*h$ is a solution of $x=Ex\vee h$. 
Indeed, if $z\geq Ez\vee h$, then we also have
\[
z\geq Ez\vee h \geq E^2 z\vee E h\vee h\geq\cdots
\geq E^{k+1}z\vee E^kh\vee E^{k-1}h\vee\cdots\vee h \enspace ,
\]
so that $z\geq E^*h$ for all such $z$. 
We sum this up in the following standard proposition.
\begin{proposition}[See e.g.~{\cite[Th.~3.17]{BCOQ}}]
\label{bellman}
Let $E\in\RRbar^{n\times n}$ and $h\in\RRbar^n$.
Then, $E^*h$ is the least solution of $z\geq Ez\vee h$.
\end{proposition}

For $E\in\RRbar^{n\times n}$, consider the associated digraph
$\digr(E)$, with set of nodes $[n]$ and an arc connecting node 
$i$ with node $j$ whenever $e_{ij}$ is finite, in which case $e_{ij}$
is the weight of this arc. 
We shall say that {\em node $i$ accesses node $j$} if there exists a
path from $i$ to $j$ in $\digr(E)$. 

The maximal (minimal) cycle mean is another important
object of tropical algebra. For
$E\in \Rmax^{n\times n}$ ($E\in \Rmin^{n\times n}$),
it is defined as
\begin{equation}\label{mcm-def}
\begin{split}
\mu^{\max}(E)&=\max_{k\in [n]}\; \max_{i_1,\ldots ,i_k} \;
\frac{e_{i_1 i_2} + \cdots + e_{i_k i_1}}{k} \quad\text{(max-plus)}\enspace , \\
\mu^{\min}(E)&=\min_{k\in [n]} \; \min_{i_1,\ldots,i_k} \;
\frac{e_{i_1 i_2} + \cdots + e_{i_k i_1}}{k} \quad\text{(min-plus)}\enspace .
\end{split}
\end{equation}
Denote by $\mu_i^{\max}(E)$ ($\mu_i^{\min}(E)$) 
the maximal (minimal) cycle mean 
of the strongly connected component of $\digr(E)$ to which $i$ belongs. 
These numbers are given by the same expressions as in~\eqref{mcm-def}, but
with $i_1,\ldots ,i_k$ restricted to that strongly connected component. 
Using $\mu^{\max}_i(E)$ ($\mu^{\min}_i(E)$), we can write explicit
expressions for the cycle-time vector of a max-plus (min-plus) 
linear function $x\mapsto Ex$:
\begin{equation}\label{chi-expl}
\begin{split}
&\chi_i^{\max}(E)=\max\{\mu_j^{\max}(E)\mid \text{$i$ accesses $j$}\}\quad \text{(max-plus)},\\
&\chi_i^{\min}(E)=\min\{\mu_j^{\min}(E)\mid \text{$i$ accesses $j$}\}\quad \text{(min-plus)}.
\end{split}
\end{equation}
See~\cite{Coc-98} or~\cite{HOW:05} for proofs. Importantly, these
cycle-time vectors of max-plus and min-plus linear functions appear
in~\eqref{e:chi-duality-games}. 
 
Finally, note that~\eqref{mcm-def} and~\eqref{chi-expl} can be deduced 
from~\eqref{val-exist} if $\sigma$ or $\tau$ is fixed. 

\begin{remark}\label{r:shortestpaths}
Observe that any entry $(i,j)$ of $E^k$ in max-plus (resp. min-plus) 
algebra expresses the maximal (resp. minimal) weight of paths with $k$ 
arcs connecting node $i$ with node $j$ in $\digr(E)$. 
It follows then from~\eqref{kls-def} that any entry $(i,j)$ of $E^*$,
for $i\neq j$, expresses the maximal (or minimal) weight of paths 
connecting node $i$ with node $j$ 
without restrictions on the number of arcs.
Further we can add to $\digr(E)$ a new node and, whenever 
$b_i$ if finite, an arc of weight $b_i$  
connecting node $i$ of $\digr(E)$ with this new node. Then,  
$(E^*b)_i$ provides the maximal (or minimal) weight of 
paths connecting node $i$ with the new node. Therefore, 
computing $E^*b$ is equivalent to solving a single destination
shortest path problem, which can be done in $O(n^3)$ time
(for instance by the Bellman-Ford algorithm). 
\end{remark}

\begin{remark}\label{r:chi-karp}
We note that $\chi^{\max}(E)$ and $\chi^{\min}(E)$ can also be computed
in $O(n^3)$ time. To do this, decompose first the digraph
$\digr(E)$ in strongly connected components, 
and apply Karp's algorithm to compute the 
maximal or minimal cycle mean of each component.
\end{remark}

\section{Tropical linear-fractional programming}

This is the main section of the paper. 
Here we solve the tropical linear-fractional programming problem~\eqref{pblp}, 
i.e.\ the problem
\begin{equation}
\label{mainproblem}
\begin{split}
&\text{minimize } \quad (px\vee r) - (qx\vee s) \\
&\text{subject to:}\quad Ax\vee c\leq Bx\vee d \; ,\;
x\in \Rmax^n 
\end{split}
\end{equation}
where $p,q\in\Rmax^n$, $c,d\in\Rmax^m$, $r,s\in\Rmax$ and  $A,B\in\Rmax^{m\times n}$. 

In Subsection~\ref{ss:Pformulations}, we apply 
Theorem~\ref{chi-axbx}
to reduce~\eqref{mainproblem} to the problem of finding
the smallest zero of a function giving the value
of a parametric game (the {\em spectral function}). 

In Subsection~\ref{ss:specf}, 
we show that the spectral function is $1$-Lipschitz and piecewise linear,
and that it can be written as a finite supremum or
infimum of {\em partial} spectral functions, 
corresponding to one player games. 
We also prove a number of technical statements about the
piecewise-linear structure of the spectral functions, which will be used in the complexity
analysis. 

In Subsection~\ref{ss:certificates}, we provide certificates of optimality
and unboundedness 
(these certificates are given by strategies for the players). 
This generalizes the result of~\cite{AGK-10}, concerning
the tropical analogue of Farkas lemma. We recover
as a special case the unboundedness certificates of~\cite{BA-08}. 

The rest of the section is devoted to finding the 
least zero of the spectral function. With this aim,  
we introduce a bisection method, as well as a Newton-type method, 
in which partial spectral functions play the role of derivatives, 
see Subsection~\ref{ss:bisnewt}. 
 
Each Newton iteration consists of
\begin{enumerate}[(i)]
\item Computing a derivative,
i.e.\ choosing a strategy for player Max (or dually Min) which satisfies
a local optimality condition;  
\item Finding the smallest zero of the tangent map, which represents
the parametric spectral function of a one-player game in which
the strategy for player Max (or dually Min) is already fixed. 
\end{enumerate} 
The iteration in the space of strategies for player Max has an advantage:
the second subproblem can be reduced to
a shortest-path problem (Subsection~\ref{ss:kleene}).
The first subproblem is discussed in Subsection~\ref{ss:left-optim}, 
where the overall worst-case complexity of Newton method is given. 
Subsection~\ref{ss:germs}, which can be skipped by the reader, 
gives an alternative approach to the first subproblem in which 
the computation of (left) optimal strategies is  
rather algebraic and not relying on the integrality.

\subsection{The spectral function method}\label{ss:Pformulations}

In this subsection we recast~\eqref{mainproblem} 
as a parametric two-sided tropical system and a mean payoff game, 
introducing the key concept of spectral function. 
However, before doing this we need to mention 
special cases in which there exists a feasible
$x$ (i.e., satisfying $Ax\vee c\leq Bx\vee d$) such that 
$px\vee r= -\infty$ or $qx\vee s= -\infty$. 
For these cases we assume the following rules:
\begin{equation}
\label{infty-rules}
\begin{array}{c|c|c}
px\vee r & qx\vee s & (px\vee r) - (qx\vee s) \\
\hline 
-\infty & \text{finite} & -\infty \\
\text{finite} & -\infty & +\infty \\
-\infty & -\infty & -\infty
\end{array}
\end{equation}
which are formally consistent with the rules of $\RRmax$. 
Then, it is easy to check that 
\begin{equation}
\label{resid-obs}
(px\vee r) - (qx\vee s) = \min\{\lambda\in \Rmax \mid px\vee r\leq \lambda + (qx\vee s)\} \; .
\end{equation}

Introducing the notation
\begin{equation}
\label{UVnotat}
U=
\begin{pmatrix}
A & c\\
p & r
\end{pmatrix}\; \makebox{ and } \;
V(\lambda )=
\begin{pmatrix}
B & d\\
\lambda +q & \lambda+s
\end{pmatrix},
\end{equation}
we reformulate the tropical linear-fractional programming problem
in terms of a spectral function, which
gives the value of a parametric
mean payoff game: the payments are given by the
matrices $U$ and $V(\lambda)$,
and the initial node is $n+1$.

\begin{propdef}
\label{pd-specf}
With assumption~\eqref{infty-rules}, 
the tropical linear-fractional programming problem~\eqref{mainproblem}
is equivalent to
\begin{equation}\label{problem-simple}
\min\{\lambda \in\Rmax\mid \spectral(\lambda)\geq 0\}
\end{equation} 
where the {\em spectral function} $\spectral$ is given by
\[ 
\spectral(\lambda):=\chi_{n+1}(U^{\sharp}V(\lambda)) \enspace .
\]
\end{propdef}
\begin{proof}
We first show that~\eqref{mainproblem} is equivalent to the following problem:
\begin{equation}\label{problem-sets}
\begin{split}
&\text{minimize } \quad \lambda \\
&\text{subject to:}\quad  px\vee r\leq \lambda+(qx\vee s)\; ,\;
Ax\vee c\leq Bx\vee d \; , \; x\in\Rmax^n \; ,\; \lambda \in\Rmax 
\end{split}
\end{equation}
Indeed, denoting $P=\left\{x\in \Rmax^n \mid Ax\vee c\leq Bx\vee d\right\}$, 
we verify that
\begin{equation*}
\begin{split}
\min\limits_{x\in P} \{ (px\vee r) - (qx\vee s)\}= 
&\min\limits_{x\in P}\min\limits_{\lambda}
\{\lambda\mid px\vee r\leq \lambda + (qx\vee s)\}\\
= &\min\limits_{\lambda}\{\exists x\in P\mid px\vee r\leq \lambda +
(qx\vee s)\} \; .
\end{split}
\end{equation*} 

Every problem concerning affine polyhedra 
has an equivalent ``homogeneous'' version concerning cones, 
which is obtained by adding to the system of inequalities 
defining an affine polyhedron a new variable whose coefficients are the free 
terms of this system. Then, the original polyhedron is recovered by 
setting this new variable to $0$. 
The homogeneous equivalent version of~\eqref{problem-sets} reads:
\begin{equation}\label{problem-cones}
\begin{split}
&\text{minimize } \quad \lambda  \\
&\text{subject to:}\quad   u y\leq \lambda + v y\; ,
\; C y\leq D y\; , \; y_{n+1}\neq -\infty \; , \; y\in\Rmax^{n+1} \; , 
\; \lambda \in\Rmax 
\end{split} 
\end{equation}
where we set $u=[p, r]$, $v=[q, s]$, $C=[A,c]$ 
and $D=[B, d]$. 

We can still reformulate~\eqref{problem-cones} in a more compact way:
\[ 
\min \{ \lambda \in\Rmax\mid U y\leq V(\lambda )y\; ,
\; y_{n+1}\neq -\infty\; \text{ is solvable}\} \; ,
\] 
with $U$ and $V(\lambda)$ defined in~\eqref{UVnotat}. 
Finally, by Theorem~\ref{chi-axbx}, it follows that 
$Uy\leq V(\lambda )y$ is solvable with finite $y_{n+1}$ if, and only if,
$\chi_{n+1}(U^{\sharp}V(\lambda ))\geq 0$.
\end{proof}

\begin{remark}
Butkovi\v{c} and Aminu~\cite{BA-08} considered
the following special cases of~\eqref{mainproblem}: 
\begin{equation}\label{problem-straight}
\begin{split}
&\text{minimize } \quad px \quad (\text{resp.\ maximize } \quad qx)\\
&\text{subject to:}\quad  Ax\vee c\leq Bx\vee d \; , \; x\in\R^n 
\end{split}
\end{equation}
where $p,q\in\R^n$, $c,d\in\R^m$, $r,s\in\R$ and  $A,B\in\R^{m\times n}$
have only finite entries. Clearly, \eqref{mainproblem} becomes~\eqref{problem-straight} if we
set  $r=-\infty$, $q\equiv -\infty$ and $s=0$ for minimization, 
or respectively $s=-\infty$, $p\equiv -\infty$ and $r=0$ for maximization, 
where the opposite (tropical inverse) of the minimal value of $\lambda$ 
equals the maximum of $qx$. 

In this connection, formulation~\eqref{problem-cones}
(or equivalently~\eqref{mainproblem})
has a good geometric insight, meaning optimization for general
tropical half-spaces (or hyperplanes) defined by bivectors $(u,\lambda +v)$, 
see~\cite{GK-09,AGK-10} for more background. 
\end{remark}

\begin{example}\label{Example1} 
Assume we want to maximize $(1+x_1)\vee (3+x_2)$ 
over the tropical polyhedron of
$\Rmax^2$ defined by the system $Ax\vee c\leq Bx\vee d$, where
\[
A=
\left(\begin{array}{cc}
-\infty & -1 \\
-2 & -2 \\
-1 & -\infty \\
0 & -\infty 
\end{array}\right) \; ,\quad 
c=
\left(\begin{array}{c}
-\infty\\
-\infty\\
-\infty\\
-\infty
\end{array}\right) \; ,\quad
B=
\left(\begin{array}{cc}
0 & -\infty \\
-\infty & -\infty \\
-\infty & 0 \\
-\infty & 2 
\end{array}\right) \; ,\quad
d=
\left(\begin{array}{c}
0\\
0\\
0\\
0
\end{array}\right) \; .
\]
This tropical polyhedron is displayed on the 
left-hand side of Figure~\ref{fig-max} below.  
This maximization problem is equivalent to minimizing
$\lambda$ subject to $0\leq \lambda + ((1+x_1)\vee (3+x_2))$,  
$Ax\vee c\leq Bx\vee d$. Indeed, 
the value of the latter problem is the opposite 
(tropical inverse) of the value of the maximization problem. 
The homogeneous version of this minimization problem reads:
\begin{equation}\label{EqProbMaxExample}
\begin{split}
&\text{minimize } \quad \lambda  \\
&\text{subject to:}\quad   u y\leq \lambda + v y\; ,
\; C y\leq D y\; , \; y_{3}\neq -\infty \; , \; y\in\Rmax^{3} \; , 
\; \lambda \in\Rmax 
\end{split} 
\end{equation}
where $C=[A,c]$, $D=[B, d]$, $u=(-\infty,-\infty,0)$, 
and $v=(1,3,-\infty)$.
\end{example}

\subsection{Partial spectral functions and piecewise linearity}\label{ss:specf}

We have shown that solving the tropical linear-fractional programming problem~\eqref{mainproblem}
is equivalent to finding the least zero of the spectral function
$\spectral(\lambda):=\chi_{n+1}(U^{\sharp}V(\lambda))$. Here we will analyze the
graph of this spectral function, after introducing analogues of derivatives, 
the partial spectral functions.

Given a strategy $\sigma\in S$ for player Max and a strategy
$\tau\in T$ for player Min, we respectively define the min-plus linear function
$U^{\sharp}V^{\sigma}(\lambda)$ and the max-plus linear function
$U_{\tau}^{\sharp}V(\lambda)$, see~\eqref{mm-only} and~\eqref{ataubsigma}. 
We introduce the {\em partial spectral functions} 
$\spectral^{\sigma}(\lambda):=\chi_{n+1}(U^{\sharp}V^{\sigma}(\lambda))$ and
$\spectral_{\tau}(\lambda):=\chi_{n+1}(U^{\sharp}_{\tau}V(\lambda))$. 
With this notation,  \eqref{e:chi-duality} yields
\begin{equation}\label{sl-duality}
\spectral(\lambda)=\max\limits_{\sigma\in S} \spectral^{\sigma}(\lambda)=
\min\limits_{\tau\in T} \spectral_{\tau}(\lambda) \; .
\end{equation}
Partial spectral functions can be represented as in~\eqref{val-exist}, 
where one of the strategies is fixed, see also~\eqref{chi-expl}:
\begin{equation}\label{phi-Phi}
\phi^{\sigma}(\lambda)=\min_{\tau\in T}\Phi_{U,V(\lambda)}(n+1,\tau,\sigma)\; ,
\quad 
\phi_{\tau}(\lambda)=\max_{\sigma\in S}\Phi_{U,V(\lambda)}(n+1,\tau,\sigma)\; .
\end{equation}

A graphical presentation of~\eqref{phi-Phi} and~\eqref{sl-duality}  using only $\spectral_{\tau}$
is given in Figure~\ref{FigureSpecFunction}.

\begin{figure}
\begin{center}
\input{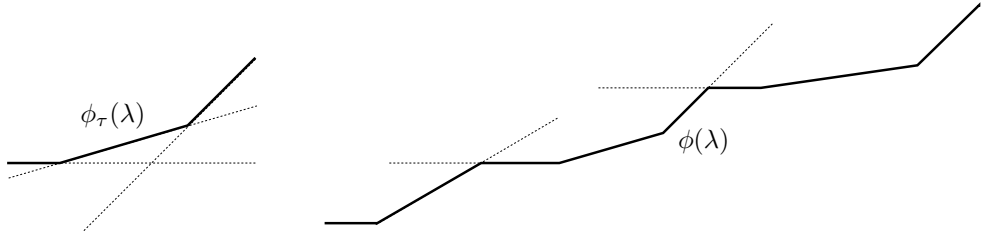}
\end{center}
\caption{A partial spectral function (left) and a spectral function (right).}
\label{FigureSpecFunction}
\end{figure}

Let $\Bipdig_\lambda $ be the bipartite digraph of the mean payoff game
whose payments are given by the matrices $U$ and $V(\lambda )$, 
see Figures~\ref{BipDigrapfExMin} and~\ref{FigureBipDigraGen} 
below for illustrations. 
Observe that in this digraph, only the weight of the arcs
connecting node $m+1$ of Max with nodes $l\in [n+1]$ of Min
depend on $\lambda $.
Recall that given a strategy $\sigma$ for player Max 
(resp.\ $\tau $ for player Min),
$\Bipdig^{\sigma}_\lambda $ (resp.\ $\Bipdig^{\tau }_\lambda $) denotes 
the bipartite subdigraph of $\Bipdig_\lambda $ obtained by deleting from
$\Bipdig_\lambda $ all the arcs $(i,j)$ such that
$i\in[m+1]$ and $j\neq\sigma(i)$
(resp.\ the arcs $(j,i)$ such that $j\in [n+1]$ and $i\neq\tau (j)$). 

We next investigate the properties of spectral functions.

\begin{theorem}\label{t:philambda}
Let $\sigma\colon [m+1]\mapsto [n+1]$ be a positional strategy for player Max 
and $\tau\colon [n+1]\mapsto [m+1]$ be a positional strategy for player Min. 
Then, 
\begin{enumerate}[(i)]
\item\label{t:philambdaP1} 
$\spectral(\lambda)$, $\spectral_{\tau}(\lambda)$ and 
$\spectral^{\sigma}(\lambda)$ are $1$-Lipschitz nondecreasing 
piecewise-linear functions, whose linear pieces are of the form 
$(\alpha+\beta\lambda)/k$, where $k\in [\min(m,n)+1]$ and
$\beta\in\{0,1\}$.
\item\label{t:philambdaP2} 
If the absolute values of all the finite coefficients in~\eqref{problem-cones} 
are bounded by $M$, then $|\alpha/k|\leq 2M$.
\item\label{t:philambdaP3} 
$\spectral_{\tau}(\lambda)$ is convex and $\spectral^{\sigma}(\lambda)$ is 
concave. Both functions consist of no more than $\min(m,n)+2$ linear pieces.
\end{enumerate}
\end{theorem}

\begin{proof}
As follows from~\eqref{sl-duality} and~\eqref{phi-Phi}, 
spectral functions are built from the finite number of 
functions $\lambda \mapsto \Phi_{U,V(\lambda)}(n+1,\sigma,\tau)$, 
each of which is given by the mean weight per turn of the only elementary 
cycle in $\Bipdig^{\sigma,\tau}_\lambda$ accessible from node $n+1$ of Min.
Recall that $\Bipdig^{\sigma,\tau}_\lambda$ is the subdigraph of 
$\Bipdig_\lambda$ where all arcs at all nodes except for 
those chosen by player Max (strategy $\sigma$) and player Min 
(strategy $\tau$) are removed. 
As a function of $\lambda$, this mean weight per turn is a line 
$(\alpha+\beta\lambda)/k$. Here $k\in [\min(m,n)+1]$ since the 
length (i.e., the number of nodes of Max it contains) 
of any elementary cycle in the bipartite digraph $\Bipdig_\lambda$ 
does not exceed both $m+1$ and $n+1$. Also $\beta\in\{0,1\}$, 
because an elementary cycle can contain node $m+1$ of Max no more than once.
If the absolute values of all the payments in the game are bounded by $M$, then
$|\alpha/k|\leq 2M$ since the arithmetic mean of payments, counted per turn, 
does not exceed the greatest sum of two consecutive payments.

Thus, the functions $\lambda \mapsto \Phi_{U,V(\lambda)}(n+1,\sigma,\tau)$ 
satisfy all the properties of~\eqref{t:philambdaP1} and~\eqref{t:philambdaP2}. 
Using~\eqref{sl-duality} and~\eqref{phi-Phi} we conclude that 
$\spectral(\lambda)$, $\spectral_{\tau}(\lambda)$ and 
$\spectral^{\sigma}(\lambda)$ also satisfy these properties. 

Convexity (resp.\ concavity) of $\spectral_{\tau}$ 
(resp.\ $\spectral^{\sigma}$) follows from~\eqref{phi-Phi}.
In a convex or concave piecewise-linear function, 
each slope can appear only once, 
while the possible slopes are $0$, $1$, $1/2$,$\ldots$, $1/(\min(m,n)+1)$. 
This shows~\eqref{t:philambdaP3}. 
\end{proof}
     
Some useful facts can be deduced further 
from this description of spectral functions.

\begin{corollary}\label{c:philambda}  
Spectral functions satisfy the following properties: 
\begin{enumerate}[(i)]
\item\label{c:philambdaP1} 
If the absolute values of all coefficients in~\eqref{problem-cones} 
are either infinite or bounded by $M$, then 
$\spectral(\lambda),\spectral^{\sigma}(\lambda)$ 
and $\spectral_{\tau}(\lambda)$ are linear for $\lambda\leq -4M(\min(m,n)+1)^2$ 
and for $\lambda\geq 4M(\min(m,n)+1)^2$.

\item\label{c:philambdaP2} 
If the absolute values of all coefficients in~\eqref{problem-cones} 
are either infinite or bounded by $M$, then
the solutions to the problems 
$\min\{\lambda \mid \spectral(\lambda)\geq 0\}$, 
$\min\{\lambda \mid \spectral^{\sigma}(\lambda)\geq 0\}$
and $\min\{\lambda \mid \spectral_{\tau}(\lambda)\geq 0\}$
lie (if finite) in $[-2M(\min(m,n)+1),2M(\min(m,n)+1)]$. 
Moreover, if all the finite coefficients are integers, 
then the solutions to all these problems are integers as well.
 
\item\label{c:philambdaP3} 
If the finite coefficients in~\eqref{problem-cones} are integers, 
then the breaking points of $\spectral(\lambda)$, 
$\spectral^{\sigma}(\lambda)$ or $\spectral_{\tau}(\lambda)$ 
are rational numbers whose denominators do not exceed $\min(m,n)+1$.

\item\label{c:philambdaP4} 
If the finite coefficients in~\eqref{problem-cones} are integers 
with absolute values bounded by $M$, 
then $\spectral(\lambda)$ consists of no more than 
$8M(\min(m,n)+1)^4+2$ linear pieces.

\end{enumerate}
\end{corollary}

\begin{proof}

\eqref{c:philambdaP1} Consider the intersection point $\mu$ of one linear piece
$(\alpha_1+\beta_1\lambda)/k_1$ with another linear piece 
$(\alpha_2+\beta_2\lambda)/k_2$. By Theorem~\ref{t:philambda}, 
$k_1,k_2\leq\min(m,n)+1$ and 
$| \alpha_1/k_1 | , | \alpha_2/k_2 | \leq 2M$,
and we obtain from
\[
| \beta_1/k_1-\beta_2/k_2 | \geq \frac{1}{(\min(m,n)+1)^2} \; ,
\quad | \alpha_1/k_1-\alpha_2/ k_2| \leq 4M \; ,
\]
that $|\mu |\leq 4M(\min(m,n)+1)^2$. 
This means that $\spectral(\lambda)$ is linear for
$\lambda\geq 4M(\min(m,n)+1)^2$ and $\lambda\leq -4M(\min(m,n)+1)^2$. 
(Note that this part did not impose the integrality of coefficients.)

\eqref{c:philambdaP2} Note that due to piecewise-linearity, 
the solution to each of these problems (if finite) is given by the intersection 
point of a certain linear piece of the form $(\alpha+\lambda)/k$ with zero. 
Then, since $|\alpha/k|\leq 2M$ and $k\leq\min(m,n)+1$ 
by Theorem~\ref{t:philambda}, we conclude that this intersection point 
$-\alpha$ lies in $[-2M(\min(m,n)+1),2M(\min(m,n)+1)]$. Moreover, 
if the finite coefficients in~\eqref{problem-cones} are integers, 
then this solution $-\alpha$ is also integer.

\eqref{c:philambdaP3} By Theorem~\ref{t:philambda}, 
spectral functions are piecewise linear and 
the linear pieces are of the form $(\alpha+\beta\lambda)/k$, 
where in particular $k\in [\min(m,n)+1]$ and $\beta\in\{0,1\}$. 
Considering the intersection point $\mu $ of one such piece
$(\alpha_1+\beta_1\lambda)/k_1$ with another piece 
$(\alpha_2+\beta_2\lambda)/k_2$ and assuming the integrity of 
$\alpha_1,\alpha_2$ we obtain that 
$\mu =(k_1 \alpha_2-k_2\alpha_1)/(k_2 \beta_1-k_1\beta_2)$ 
is a rational number with denominator not exceeding $\min(m,n)+1$.

\eqref{c:philambdaP4} 
The denominators of breaking points do not exceed $\min(m,n)+1$, 
and hence the difference between their inverses is not less than 
$1/(\min(m,n)+1)^2$. This is a lower bound for the difference between two
consecutive breaking points. 
We get the claim applying part~\eqref{c:philambdaP1} .
\end{proof}

Note that to determine the slope of $\phi(\lambda)$ at $+\infty$, 
meaning for $\lambda\geq 4M(\min(m,n)+1)^2$, or at $-\infty$, 
meaning for $\lambda\leq -4M(\min(m,n)+1)^2$, 
we can set all the finite coefficients in~\eqref{problem-cones} to $0$. 
Then, we ``play'' the mean payoff game at 
$\lambda=1$ or at $\lambda=-1$, respectively. 

Denote by $\MPGinteger(m,n,M)$ the worst-case complexity of an oracle 
computing the {\em value} of mean payoff games with integer payments whose
absolute values are bounded by $M$, with $m$ nodes of Max and $n$ nodes of Min. 
There exist pseudo-polynomial algorithms computing the value of mean payoff 
games. For instance, in~\cite{ZP-96} the authors describe
a value iteration algorithm with $O(mn^4M)$ complexity.
Using this we now show that all the linear pieces of 
a spectral function can be identified in pseudo-polynomial time. Note that we do not require
the oracle to compute optimal strategies here.

\begin{proposition}\label{p:reconstr} 
Let all the finite coefficients in~\eqref{problem-cones}
be integer with absolute values not exceeding $M$. Then,  
all the linear pieces that constitute the graph of
$\spectral(\lambda)$ can be identified in 
\[
O(M\min(m,n)^4)\times \MPGinteger(m+1,n+1,M(\min(m,n)+1)(1+4(\min(m,n)+1)^2))
\] 
operations. 
\end{proposition}

\begin{proof}
By Corollary~\ref{c:philambda} part~\eqref{c:philambdaP3}, 
the breaking points of $\spectral(\lambda)$ are rational numbers 
whose denominators do not exceed $\min(m,n)+1$. 
To identify the linear pieces that constitute the graph of
the spectral function, we only need to evaluate $\spectral(\lambda)$  
on such rational points  in the interval 
$[-4M(\min(m,n)+1)^2, 4M(\min(m,n)+1)^2 ]$, 
the number of which does not exceed $O(M\min(m,n)^4)$.

Further, when computing $\spectral(\lambda)$, 
the payments in the mean payoff games that the oracle works with 
are either $a$ or $a+\lambda$, where $a$ is an integer satisfying 
$|a|\leq M$ and $\lambda$ is a rational number in 
$[-4M(\min(m,n)+1)^2, 4M(\min(m,n)+1)^2 ]$
whose denominator does not exceed $\min(m,n)+1$.
The properties of the game will not change if we multiply 
all the payments by this denominator, obtaining a new game 
in which the payments are integers with absolute values bounded by 
$(\min(m,n)+1)(M+4M(\min(m,n)+1)^2)$. Then, 
the complexity of the mean payoff oracle will not exceed 
$\MPGinteger(m+1,n+1,(\min(m,n)+1)(M+4M(\min(m,n)+1)^2))$. Multiplying by
$O(M\min(m,n)^4)$ we get the claim. 
\end{proof}

\begin{remark}\label{r:pseudopol}
It follows that the tropical linear-fractional programming problem~\eqref{problem-cones}
can be solved in pseudo-polynomial time by reconstructing all the 
linear pieces that constitute the graph of $\spectral(\lambda)$.  However,
more efficient methods will be described in Subsection~\ref{ss:bisnewt}.
\end{remark}

\begin{remark}\label{r:parmpg}
A similar spectral function has been introduced in~\cite{GStwosided10} to compute the set of solutions $\lambda$ of the two-sided eigenproblem $Ax=\lambda Bx$.
The present approach can be extended to a larger
class of parametric games, in which the payments are piecewise
affine functions of the parameter $\lambda$, with integer
slopes. See~\cite{Ser-lastdep}.
\end{remark}

\subsection{Strategies as certificates}\label{ss:certificates}
In the classical simplex method, the optimality of a feasible
solution is certified by the sign of Lagrange multipliers. 
In the tropical case, following the idea of~\cite{AGK-10},
we shall show that the certificate is of a different nature: 
it is a strategy.
We shall also use such strategies to guide the next iteration 
of Newton method in Subsection~\ref{ss:bisnewt}, 
when the current feasible solution is not optimal.

\begin{definition}[Left and right optimal strategies]
\label{d:lropt}
A strategy $\sigma-$ for player Max (resp.\ $\tau-$ for player Min) 
is {\em left optimal} at $\lambda\in\R$, 
if there exists $\epsilon>0$ such that
\[ 
\spectral(\mu )=\spectral^{\sigma-}(\mu ) 
\qquad (\makebox{resp.\ }\spectral(\mu ) =\spectral_{\tau-}(\mu)) 
\qquad \forall \mu \in [\lambda -\epsilon ,\lambda ] \enspace.
\]
Right optimal strategies $\sigma+$ and $\tau+$ are defined in a similar way, 
replacing $[\lambda -\epsilon ,\lambda ]$ by $[\lambda ,\lambda +\epsilon ]$.
\end{definition}

The existence of left and right optimal strategies
at each point follows readily from~\eqref{sl-duality},
together with the finiteness of the number of strategies
and the piecewise affine character of each function 
$\phi^\sigma(\lambda)$ and $\phi_\tau(\lambda)$. 

\begin{theorem}\label{1st-cert}
The tropical linear-fractional programming problem~\eqref{problem-cones}
has the optimal value $\lambda^\ast \in \R$ if, and only if,
$\spectral(\lambda^\ast)\geq 0$ and there exists a strategy
$\tau $ for player Min such that the digraph
$\Bipdig_{\lambda^\ast }^{\tau}$ satisfies the following conditions:
\begin{enumerate}[(i)]
\item all cycles accessible from node $n+1$ of Min have nonpositive weight, 
\item any cycle of zero weight accessible from node $n+1$ of Min
passes through node $m+1$ of Max.
\end{enumerate}
Moreover, these conditions are always satisfied 
when $\tau$ is left optimal at $\lambda^\ast $.
\end{theorem}

\begin{proof}
The tropical linear-fractional programming problem~\eqref{problem-cones}
has the optimal value $\lambda^\ast $ if, and only if,
$\spectral(\lambda^\ast )=0$ and $\spectral(\lambda)<0$ for all $\lambda <\lambda^\ast $.
If $\tau$ is any left optimal strategy at $\lambda^*$,
then the previous conditions are satisfied if, and only if,
$\spectral_{\tau}(\lambda^\ast )=0$ and $\spectral_{\tau}(\lambda)$ has nonzero
left derivative at $\lambda^*$.

By~\eqref{phi-Phi}, or~\eqref{mcm-def} and~\eqref{chi-expl}, we know that 
$\spectral_{\tau}(\lambda)$ 
is the maximal cycle mean (per turn) over all cycles in $\Bipdig^{\tau}_\lambda$
accessible from node $n+1$ of Min.
It follows that
$\spectral_{\tau}(\lambda^\ast )=0$ if, and only if,
all cycles in $\Bipdig^{\tau}_{\lambda^\ast  }$
accessible from node $n+1$ of Min have nonpositive weight 
and at least one of them has zero weight. Moreover,
$\spectral_{\tau}(\lambda )$ has nonzero left derivative at $\lambda^\ast $ if,
and only if, any zero-weight cycle in
$\Bipdig^{\tau}_{\lambda^\ast  }$ accessible from node $n+1$ of Min has
arcs with weights depending on $\lambda $,
which can only occur if it passes through node $m+1$ of Max.
Thus, the conditions of the theorem are necessary and they are satisfied by 
any left optimal strategy $\tau $ at $\lambda^\ast $.

Assume now that there exists a strategy $\tau$ satisfying the conditions
of the theorem. Then, the argument above shows that
$\spectral_\tau(\lambda^\ast)\leq 0$ and $\spectral_\tau(\lambda)<0$
for all $\lambda < \lambda^\ast$. Since $\spectral(\lambda^\ast)\geq 0$ and
by~\eqref{sl-duality} we have $\spectral(\lambda)\leq \spectral_\tau(\lambda)$
for all $\lambda$, it follows that $\spectral(\lambda^\ast)=0$ and $\spectral(\lambda)<0$
for all $\lambda < \lambda^\ast$. Therefore,
$\lambda^\ast $ is the optimal value of
the tropical linear-fractional programming problem~\eqref{problem-cones}.
\end{proof}

In the same way, we can certify when
the tropical linear-fractional programming problem~\eqref{problem-cones} is unbounded.

\begin{theorem}\label{2nd-cert}
The tropical linear-fractional programming problem~\eqref{problem-cones} is unbounded if,
and only if, there exists a strategy $\sigma$ for player Max such that all 
cycles in the digraph $\Bipdig^{\sigma}_0$ accessible from node $n+1$ 
of Min do not contain node $m+1$ of Max and have nonnegative weight.
\end{theorem}

\begin{proof}
We know that the tropical linear-fractional programming problem~\eqref{problem-cones} 
is unbounded if, and only if, $\spectral(\lambda)\geq 0$ for all $\lambda$.
By the first equality in~\eqref{sl-duality},
the latter condition is satisfied if, and only if,
there exists a strategy $\sigma$ for player Max such that
$\spectral^{\sigma}(\lambda)\geq 0$ for all $\lambda$.
Note that the weight of a cycle in $\Bipdig^{\sigma}_\lambda $
that passes through node $m+1$ of Max
can be made arbitrarily small by decreasing $\lambda $,
because this cycle must contain an arc whose weight depends on $\lambda $.
Therefore, using the fact that
$\spectral^{\sigma}(\lambda)$ %
is the minimal cycle mean (per turn) over all cycles in
$\Bipdig^{\sigma}_{\lambda}$ accessible from node $n+1$ of Min
(see~\eqref{phi-Phi}, or~\eqref{mcm-def} and~\eqref{chi-expl}),
it follows that $\spectral^{\sigma}(\lambda)\geq 0$ for all $\lambda$ if, and only if,
all cycles in $\Bipdig^{\sigma}_0$
accessible from node $n+1$ of Min have nonnegative weight and
do not pass through node $m+1$ of Max.
\end{proof}
\begin{remark}
Theorems~\ref{1st-cert} and~\ref{2nd-cert} are inspired by Theorem~18 and Corollary~20 of~\cite{AGK-10}, in which similar certificates are given
for the problem of checking whether an implication of the form
$Ax\leq Bx\implies px \leq qx$ holds.
The latter can be cast as a special tropical linear-fractional
programming problem.
\end{remark}
\begin{example}\label{Example2}
Consider the tropical linear programming problem 
given by the minimization of $(2+x_1)\vee (-4+x_2)$ 
over the tropical polyhedron of $\Rmax^2$ defined
by the system of inequalities $Ax\vee c\leq Bx\vee d$, 
where
\[
A=
\left(\begin{array}{cc}
-\infty & -\infty \\
-\infty & -\infty \\
-\infty & -\infty \\
-\infty & -3 \\
-\infty & -4 \\
-\infty & -5 \\
-\infty & -6 
\end{array}\right) \; , \quad
c=
\left(\begin{array}{c}
0\\
0\\
0\\
0\\
-\infty\\
-\infty\\
-\infty
\end{array}\right) \; ,\quad
B=
\left(\begin{array}{cc}
-2 & 0 \\
0 & -1 \\
1 & -2 \\
2 & -\infty \\
0 & -\infty \\
-2 & -\infty \\
-4 & -\infty 
\end{array}\right) \; , \quad
d=
\left(\begin{array}{c}
-\infty\\
-\infty\\
-\infty\\
-\infty\\
0\\
0\\
0
\end{array}\right) \; .
\]
This polyhedron is displayed on the left-hand side of 
Figure~\ref{tropprogs-ex} below.
The direction of minimization of $(2+x_1)\vee (-4+x_2)$ 
is shown there by a dotted line above the polyhedron, 
together with the optimal tropical hyperplane $(2+x_1)\vee (-4+x_2)=0$. 
The bipartite digraph $\Bipdig_\lambda $
corresponding to this problem
is depicted in Figure~\ref{BipDigrapfExMin}, where the nodes of Max
are represented by squares and the nodes of Min by circles.
Note that in this case we have $m=7$ and $n=2$.

The equivalent homogeneous version of this problem 
(as described in Subsection~\ref{ss:Pformulations}) 
is to minimize $\lambda $ subject to 
$u y\leq \lambda + v y$, $C y \leq D y$, and $y_3\neq -\infty$, 
where $C=[A,c]$, $D=[B,d]$, $u=(2,-4,-\infty)$ and $v=(-\infty,-\infty,0)$. 

\begin{figure}
\begin{center}
\input{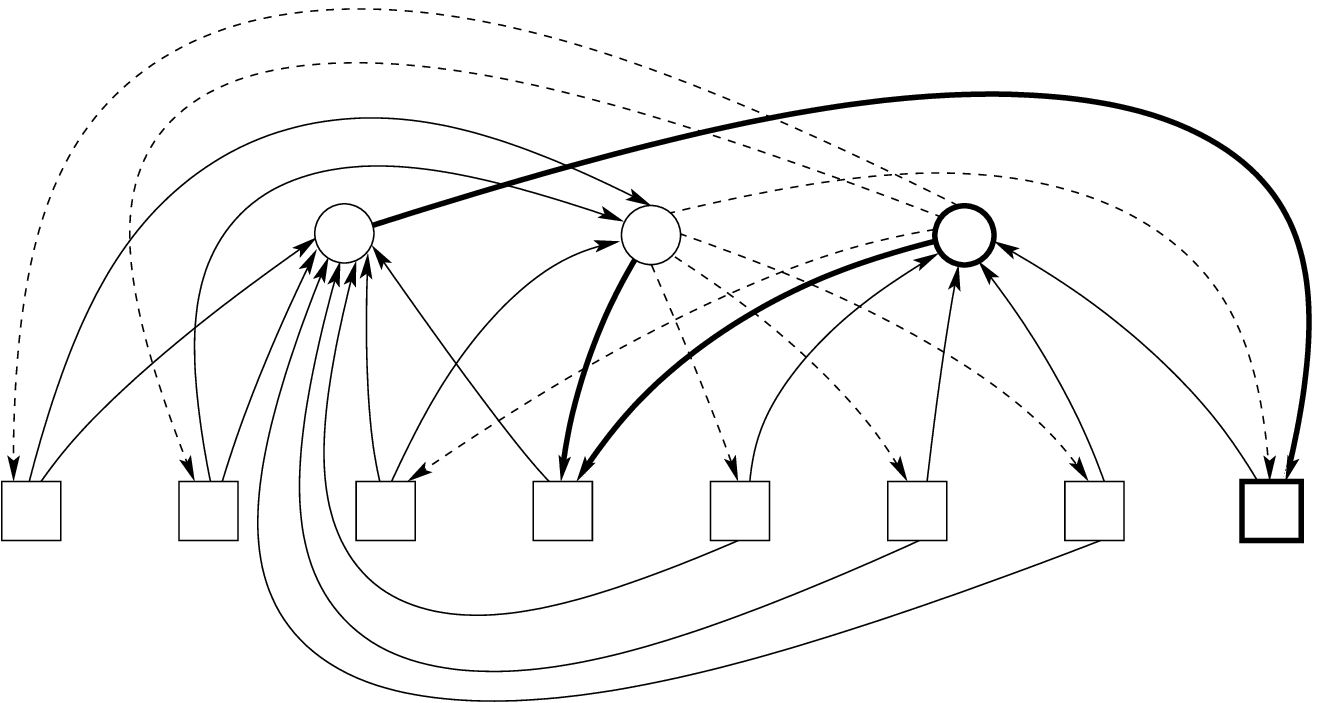}
\end{center}
\caption{The bipartite digraph $\Bipdig_\lambda$ of the mean payoff game 
associated with the tropical linear programming problem of 
Example~\ref{Example2}. 
Special nodes $n+1$ and $m+1$ are in bold as well as a strategy $\tau $ 
for player Min certifying the optimality of $\lambda^\ast =0$. 
The subdigraph $\Bipdig^\tau_{\lambda^\ast}$ 
is obtained by deleting the dashed arcs.}
\label{BipDigrapfExMin}
\end{figure}

Thanks to Theorem~\ref{1st-cert}, it is possible to certify 
that $\lambda^\ast=0$ is the optimal value of this problem. 
To show this, consider the strategy $\tau $ for player Min defined by: 
$\tau (1)=8$, $\tau(2)= 4$ and $\tau(3)=4$, which is  
represented in bold in Figure~\ref{BipDigrapfExMin}. 
Observe that the resulting subdigraph $\Bipdig_{\lambda^\ast }^{\tau}$ 
contains only one cycle, which is accessible from node $n+1$ of Min 
(indeed it passes through this node), 
has zero weight and passes through node $m+1$ of Max. 
Moreover, by Theorem~\ref{chi-axbx}, we have 
$\spectral(\lambda^\ast)\geq 0$ because $y=(-2,2,0)^T$ 
satisfies $C y \leq D y$ and $u y\leq \lambda^\ast + v y= v y$. 
Therefore, by Theorem~\ref{1st-cert}, $\lambda^\ast=0$ is the optimal value.  
\end{example}

The special cases~\eqref{problem-straight} of the tropical linear-fractional  
programming problem~\eqref{problem-cones} have been studied in~\cite{BA-08}, 
where necessary and sufficient conditions for these 
problems to be unbounded were in particular given. 
We next show that under the assumptions of~\cite{BA-08}, 
which require the entries of all vectors and matrices to be finite, 
these conditions turn out to be equivalent to the one given in 
Theorem~\ref{2nd-cert}. 

Theorem~3.3 of~\cite{BA-08} shows that, when only finite entries are 
considered, the minimization problem in~\eqref{problem-straight} 
is unbounded if, and only if, $c\leq d$. Under the finiteness assumption, 
this condition is equivalent to the one given in Theorem~\ref{2nd-cert}. 
To show this, in the first place observe that in this case 
the associated digraph $\Bipdig_\lambda $ 
(see Figure~\ref{FigureBipDigraGen}) 
contains arcs connecting any node of Min $[n+1]$ with any the node of 
Max $[m+1]$, with exception of the arc connecting node $n+1$ with 
node $m+1$, and arcs connecting any node of Max $[m+1]$ 
with any node of Min $[n+1]$, with exception of the arcs connecting 
node $m+1$ with nodes in $[n]$. Thus, if we define the strategy 
$\sigma $ for player Max by $\sigma (i)=n+1$ for all $i\in [m+1]$, 
it can be checked that the only cycles in $\Bipdig^{\sigma}_0 $ 
accessible from node $n+1$ are of the form 
$n+1\rightarrow i_1 \rightarrow n+1 \rightarrow \cdots \rightarrow n+1 \rightarrow i_k \rightarrow n+1$ for some $i_1,\ldots ,i_k \in [m]$. 
Since the weight of such a cycle is $d_{i_1}-c_{i_1}+\cdots +d_{i_k}-c_{i_k}$, 
the strategy $\sigma $ satisfies the 
conditions in Theorem~\ref{2nd-cert} if $c\leq d$. 
Conversely, assume that a strategy $\sigma $ for player 
Max satisfies the conditions in Theorem~\ref{2nd-cert}. Then, 
the only possible value for $\sigma (m+1)$ is $n+1$, 
and we must also have $\sigma (i)=n+1$ for all $i\in [m]$,
because if $\sigma (i)=j\neq n+1$ for some $i\in [m]$, 
$\Bipdig^{\sigma}_0 $ would contain the cycle 
$m+1\rightarrow n+1\rightarrow i\rightarrow j\rightarrow m+1$, 
contradicting the fact that no cycle accessible from node 
$n+1$ of Min passes through node $m+1$ of Max. Now, 
since $\Bipdig^{\sigma}_0 $ contains the cycles 
$n+1\rightarrow i \rightarrow n+1$ for $i\in [m]$, 
which are accessible from node $n+1$ of Min, 
the weights of these cycles $d_i-c_i$ must be nonnegative, 
implying that $c\leq d$. 

\begin{figure}
\begin{center}
\input{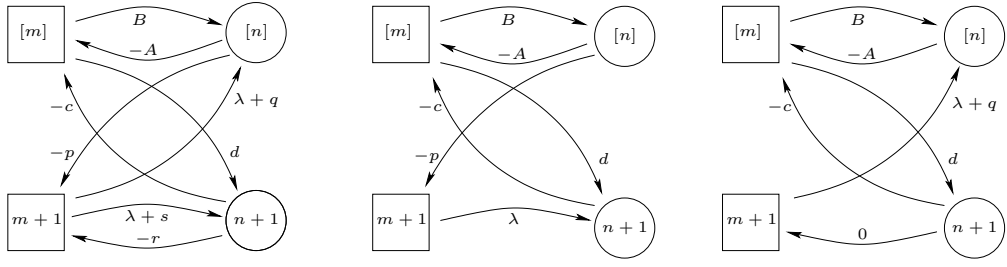}
\end{center}
\caption{On the left: the bipartite digraph of the mean payoff game associated 
with the tropical linear programming problem~\eqref{problem-sets}. 
The nodes of Max are represented by squares and the nodes of 
Min by circles. In the middle and on the right: 
the bipartite digraphs of the special cases~\eqref{problem-straight}, 
for the minimization problem and the maximization problem, 
respectively.}
\label{FigureBipDigraGen}
\end{figure}

Regarding the maximization problem in~\eqref{problem-straight}, 
Theorem 3.4 of~\cite{BA-08} shows that this problem is unbounded if, 
and only if, the system $A x \leq B x$ has a finite solution.  
In this case, due to the finiteness assumption, it follows that 
the associated digraph $\Bipdig_\lambda $ 
(see Figure~\ref{FigureBipDigraGen}) 
contains arcs connecting any node of Max $[m+1]$ 
with any node of Min $[n+1]$, with exception of the arc 
connecting node $m+1$ with node $n+1$, 
and arcs connecting any node of Min $[n+1]$ with any the node of 
Max $[m+1]$, with exception of the arcs connecting nodes in $[n]$ 
with node $m+1$. If the system $A x \leq B x$ has a finite solution, 
from Theorem~\ref{chi-axbx} and~\eqref{e:chi-duality} 
it follows that there exists a strategy 
$\bar{\sigma}: [m]\mapsto [n]$ such that 
$\chi(A^{\sharp}B^{\bar{\sigma}}) = \chi(A^{\sharp}B)\geq 0$. 
By~\eqref{mcm-def} and~\eqref{chi-expl}, 
this implies that any cycle in $\bar{\Bipdig}^{\bar{\sigma}}$ 
has nonnegative weight, where $\bar{\Bipdig}$ is the bipartite 
digraph of the mean payoff game associated with the matrices $A$ and $B$. 
If we define the strategy $\sigma(i)=\bar{\sigma}(i)$ for all $i\in [m]$ 
and $\sigma(m+1)=j$ for some $j\in [n]$, 
then $\sigma $ satisfies the conditions of Theorem~\ref{2nd-cert} 
because the cycles accessible from node $n+1$ of Min in $\Bipdig^{\sigma}_0 $ 
are precisely the cycles in $\bar{\Bipdig}^{\bar{\sigma}}$ 
and there is no cycle containing node $m+1$ of Max in $\Bipdig^{\sigma}_0 $. 
Conversely, if a strategy $\sigma $ for player Max satisfies the conditions 
in Theorem~\ref{2nd-cert}, then necessarily we have $\sigma (i)\in [n]$ 
for all $i\in [m]$, because if $\sigma (i)=n+1 $ for some $i\in [m]$, 
$\Bipdig^{\sigma}_0 $ would contain the cycle 
$n+1\rightarrow m+1 \rightarrow j \rightarrow i \rightarrow n+1$  
where $j=\sigma (m+1)\in [n]$, contradicting the fact that there is no cycle 
in $\Bipdig^{\sigma}_0 $ accessible from node $n+1$ of Min passing through node 
$m+1$ of Max. Now, if we define $\bar{\sigma}(i)=\sigma(i)$ for all $i\in [m]$, 
the cycles accessible from node $n+1$ of Min in $\Bipdig^{\sigma}_0 $ 
are precisely the cycles in $\bar{\Bipdig}^{\bar{\sigma}}$, 
which therefore have nonnegative weight. Then, 
by~\eqref{mcm-def} and~\eqref{chi-expl} we have 
$\chi(A^{\sharp}B^{\bar{\sigma}})\geq 0$, 
and so from Theorem~\ref{chi-axbx} and~\eqref{e:chi-duality} 
we conclude that the system $A x\leq Bx $ has a finite solution.         

\begin{remark}
If the strategies $\sigma$ or $\tau$ and the scalar $\lambda^*$ are fixed
(considered as inputs)
the conditions of Theorems~\ref{1st-cert} and~\ref{2nd-cert}, 
i.e.\ the validity of the certificates, 
can be checked in polynomial time. 

To see this, in the first place assume that $\tau$ and $\lambda^*$ are given. 
Using Karp's algorithm, compute the maximal cycle mean of 
each strongly connected component of 
$\Bipdig_{\lambda^\ast }^{\tau}$ that is accessible from node $n+1$ of Min. 
The certificate is valid only if these maximal cycle means are 
nonpositive and one of them is zero. To check the second condition of 
Theorem~\ref{1st-cert}, delete node $m+1$ of Max (and the arcs adjacent 
to it) from $\Bipdig_{\lambda^\ast }^{\tau}$ and compute 
for the resulting digraph (using again Karp's algorithm) 
the maximal cycle mean of each strongly connected component 
accessible from node $n+1$ of Min. To be valid, 
all these maximal cycle means must be negative. 
Observe that in Theorem~\ref{1st-cert} 
we also assume that $\phi(\lambda^\ast )\geq 0$.  
By~\eqref{sl-duality}, this can be certified by a 
strategy $\sigma $ for player Max such that the 
minimal cycle mean of any strongly connected component 
of $\Bipdig^{\sigma}_{\lambda^\ast }$ accessible from node $n+1$ of Min 
is nonnegative, which can be checked by applying Karp's 
algorithm to each of these components. By Theorem~\ref{chi-axbx}, 
another possibility is to exhibit a vector $y$ such that 
$C y \leq D y$, $u y\leq \lambda^\ast + v y$ and $y_{n+1}\neq -\infty$. 

Assume now that $\sigma$ is given. 
To check the validity of the certificate in Theorem~\ref{2nd-cert}, 
decompose first $\Bipdig^{\sigma}_0$ in strongly connected components  
and see whether the component containing node $m+1$ of Max is trivial 
(i.e.\ contains just this node) or it is not accessible from node $n+1$ of Min. 
If this is the case, compute the minimal cycle mean of each 
strongly connected component of $\Bipdig^{\sigma}_0$ 
accessible from node $n+1$ of Min by applying 
Karp's algorithm. Then, the certificate is valid if each of 
these minimal cycle means is nonnegative.

\end{remark}

\subsection{Bisection and Newton methods for tropical linear-fractional programming}\label{ss:bisnewt}

In~\eqref{problem-simple}, we need to find the least
$\lambda $ such that $\spectral(\lambda )\geq 0$, where $\spectral(\lambda )$
is nondecreasing and Lipschitz continuous.
Thus, we can consider certain classical
methods for finding zeroes of ``good enough'' functions of one variable. 
In particular, the bisection method for $\spectral(\lambda)$ 
corresponds to the approach of~\cite{BA-08}. More specifically, 
it can be formulated as follows, 
when the finite coefficients in~\eqref{problem-cones} are integers.

\begin{myalgorithm}\label{a:bisection}
Bisection method 

{\bf Start.} A point $\overline{\lambda}_0$ such that 
$\spectral(\overline{\lambda}_0)\geq 0$
and a point $\underline{\lambda}_0$ such that 
$\spectral(\underline{\lambda}_0)< 0$.

{\bf Iteration $k$.} Let  
$\lambda =\lceil(\overline{\lambda}_{k-1}+\underline{\lambda}_{k-1})/2\rceil$.
If $\spectral(\lambda)\geq 0$, then set 
$\overline{\lambda}_{k}=\lambda$ and 
$\underline{\lambda}_{k}=\underline{\lambda}_{k-1}$.
Otherwise, set $\overline{\lambda}_{k}=\overline{\lambda}_{k-1}$ 
and $\underline{\lambda}_{k}=\lambda$.

{\bf Stop.} Verify $\overline{\lambda}_{k}-\underline{\lambda}_{k}=1$. 
If true, return $\overline{\lambda}_{k}$.
\end{myalgorithm}

For this method, which uses that tropical linear-fractional programming preserves 
integrity (Corollary~\ref{c:philambda} part~\eqref{c:philambdaP4}),
it is not important to know the actual value of $\spectral(\lambda)$, 
but just whether $\spectral(\lambda)\geq 0$, 
i.e.\ whether $Uy\leq V(\lambda)y$ is solvable with $y_{n+1}\neq -\infty$. 

Further, the concept of (left, right) optimal strategy, 
see Definition~\ref{d:lropt}, 
yields an analogue of (left, right) derivative, 
and leads to the following analogue of Newton method, 
which does not have any integer restriction.

\begin{myalgorithm}\label{a:pos-newton}
Positive Newton method 

{\bf Start.} A point $\lambda_0$ such that $\spectral(\lambda_0)\geq 0$. 

{\bf Iteration $k$.} Find a left optimal strategy 
$\sigma $ for player Max at $\lambda_{k-1}$ and compute 
$\lambda_k=\min\{\lambda \in\Rmax \mid \spectral^{\sigma}(\lambda)\geq 0\}$. 

{\bf Stop.} Verify $\lambda_k=\lambda_{k-1}$ or $\lambda_k=-\infty$. 
If true, return $\lambda_{k}$.
\end{myalgorithm}

It remains to explain how each step of this algorithm can be implemented. 
We shall see that $\lambda_k$ can be easily computed
(reduction to a shortest path problem) and that finding
left optimal strategies can be done by existing
algorithms for mean payoff games. 

For the sake of comparison, we state a dual
version of Algorithm~\ref{a:pos-newton}.

\begin{myalgorithm}\label{a:neg-newton}
Negative Newton method  

{\bf Start.} A point $\lambda_0$ such that $\spectral(\lambda_0)< 0$.  

{\bf Iteration $k$.} Find a (right) optimal strategy $\tau$ for player Min 
at $\lambda_{k-1}$ and compute 
$\lambda_k=\min\{\lambda \in\Rmax \mid \spectral_{\tau}(\lambda)\geq 0\}$. 

{\bf Stop.} Verify $\spectral(\lambda_k)=0$ or $\lambda_k=+\infty$. 
If true, return $\lambda_{k}$.
\end{myalgorithm}

\begin{remark}\label{r:justopt}
Note that in Algorithm~\ref{a:neg-newton} we can use
optimal strategies instead of right optimal ones, 
because if $\tau$ is optimal at $\lambda_{k-1}$, 
we have $\spectral_{\tau}(\lambda_{k-1})=\spectral(\lambda_{k-1})<0$  
and so $\lambda_{k}> \lambda_{k-1}$ by the definition of $\lambda_{k}$ 
(recall that by Theorem~\ref{t:philambda} 
all the spectral functions are nondecreasing and piecewise-linear). 
This means that all the strategies considered in the iterations 
of Algorithm~\ref{a:neg-newton} are different, 
and as the number of strategies is finite, 
this algorithm must terminate in a finite number of steps. 
A similar argument shows that we can also use optimal strategies 
in Algorithm~\ref{a:pos-newton} at points $\lambda_{k-1}$ 
where the spectral function 
$\spectral$ is strictly positive, 
because in that case we have $\lambda_{k} < \lambda_{k-1}$ 
even if $\sigma $ is just optimal and not left optimal 
(however, when $\spectral(\lambda_{k-1})=0$ and $\lambda_{k-1}$ 
is not optimal, only a left optimal strategy at $\lambda_{k-1}$ 
guarantees $\lambda_{k} < \lambda_{k-1}$). 
\end{remark}

\begin{remark}\label{r:newtstart}
Due to Corollary~\ref{c:philambda} part~\eqref{c:philambdaP2}, 
the values $\lambda^+:=2M(\min(m,n)+1)$ and
$\lambda^-:=-2M(\min(m,n)+1)$ can be first checked in the case 
of the positive and negative Newton methods, respectively. We recall that
$M$ is a bound on the absolute value of the coefficients in~\eqref{problem-cones}.

If $\spectral(\lambda^+)<0$ then the problem is infeasible,
and if $\spectral(\lambda^-)>0$ then the problem is unbounded.
If $\spectral(\lambda^+)\geq 0$
and $\spectral(\lambda^-)<0$, then the problem is both feasible and bounded. 
The case $\spectral(\lambda^-)=0$ requires a left optimal strategy for player 
Max at $\lambda^-$ to decide that either this point is optimal, 
or the problem is unbounded.

\if{
Then, we verify $\spectral(\lambda^+)\geq 0$ in the positive Newton method 
and $\spectral(\lambda^-)<0$ in the negative one. If true, 
the algorithm proceeds normally, and if not, 
the problem can be immediately solved by finding 
the intersection with zero of the linear piece of the 
spectral function $\spectral$ at $\lambda^++1$ or $\lambda^--1$. 
This can be done by computing optimal strategies at these points, 
which by Corollary~\ref{c:philambda} part~\eqref{c:philambdaP2} are 
right and left optimal, respectively.  
This remark also applies to the bisection method. 

Observe that part~\eqref{c:philambdaP2} of Corollary~\ref{c:philambda} also 
implies the positive (resp.\ negative) Newton method can be stopped if the 
value $-4M(\min(m,n)+1)^2$ (resp.\ $4M(\min(m,n)+1)^2$) is reached, 
because in that case the last possible linear piece 
of $\spectral(\lambda)$ is being considered. 
}\fi

This rule of starting with $\pm 2M(\min(m,n)+1)$, as we shall see, 
secures pseudo-polynomiality of the instances of the mean payoff games 
generated by the bisection and Newton methods. 
\end{remark}

The following logarithmic bound on the complexity of the 
bisection method is standard and its proof will be omitted.

\begin{proposition}\label{termination-bis}
If the finite coefficients in~\eqref{problem-cones} are integers 
with absolute values bounded by $M$,
then the number of iterations of the bisection method does not exceed 
$\log(4M(\min(m,n)+1))$ if it is started as in Remark~\ref{r:newtstart}. 
Hence, the computational complexity of the bisection method in this case 
does not exceed 
\[
\log(4M(\min(m,n)+1))\times\MPGinteger(m+1,n+1,M+2M(\min(m,n)+1))\; .
\]
\end{proposition}

\begin{remark}\label{r:bisbounds}
Butkovi\v{c} and Aminu~\cite{BA-08} give better initial 
values $\overline{\lambda}_0$ and $\underline{\lambda}_0$ 
for the bisection method than $\pm 2M(\min(m,n)+1)$,  
but only for the special cases~\eqref{problem-straight}, 
where all the coefficients are assumed to be finite. 
These initial values depend on the input data and lie in the interval $[-3M,3M]$. 
As oracle, they exploit the alternating method of~\cite{CGB-03}, 
which requires $O(mn(m+n)M)$ operations, 
being related to the value iteration of~\cite{ZP-96}. 
Hence, in this case, the complexity of the bisection method
is no more than $O(mn(m+n)M \log M)$. In~\cite{Ser-lastdep}, 
the same kind of initial values were obtained for the general
formulation~\eqref{mainproblem},
leading to a similar complexity, 
but with the same finiteness restriction on the coefficients. 
The initial values of~\cite{BA-08} and~\cite{Ser-lastdep}  
will be exploited in the numerical experiments, 
see Subsection~\ref{ss:NumExp}.
\end{remark}

As observed above, in the case of the bisection method the mean payoff oracle 
is only required to check whether $\spectral(\lambda)\geq 0$. 

In the case when the finite coefficients in~\eqref{problem-cones} are real, 
the bisection method computes 
$(\overline{\lambda}_{k-1}+\underline{\lambda}_{k-1})/2$
without rounding and it yields only an approximate solution to the problem. 
However, Newton methods always converge in a finite number of steps.

\begin{proposition}\label{termination-newt}
Denote by $\stratmax$ and $\stratmin$ the number 
of available strategies for players Max and Min, 
respectively. Then, 
\begin{enumerate}[(i)]
\item\label{termination-newtP1} 
Algorithms~\ref{a:pos-newton} and~\ref{a:neg-newton} 
terminate in a finite number of steps, 
the number of which does not exceed 
$\stratmax$ and $\stratmin$, respectively.
\item\label{termination-newtP2} 
If the finite coefficients in~\eqref{problem-cones} are integers 
with absolute values bounded by $M$, 
then the values $\lambda_k$ produced by Algorithm~\ref{a:pos-newton}  
are also integer, and the number of iterations 
does not exceed $4M(\min(m,n)+1)$ 
if it is started as in Remark~\ref{r:newtstart}.
\end{enumerate}
\end{proposition}

\begin{proof}
\eqref{termination-newtP1} 
At different iterations of Algorithm~\ref{a:pos-newton} 
we have different strategies, because 
$\lambda_k=\min\{\lambda \in\Rmax \mid \spectral^{\sigma}(\lambda)\geq 0\}$ 
are different for all $k$.
Similarly for Algorithm~\ref{a:neg-newton}, 
with $\tau$ instead of $\sigma$. Thus, 
the number of steps is limited by the number of strategies, which is finite. 

\eqref{termination-newtP2} 
The numbers $\lambda_k$ generated by Algorithm~\ref{a:pos-newton} are 
integers due to Corollary~\ref{c:philambda} part~\eqref{c:philambdaP2}. 
By Remark~\ref{r:newtstart}, we can start the algorithm at 
$2M(\min(m,n)+1)$, and it will finish before it reaches $-2M(\min(m,n)+1)$.
\end{proof}

A worst-case complexity bound, different from the number of strategies, 
will be given below in Theorem~\ref{posnewt-comp},  
and it is worse than that of 
the bisection method above (see also Remark~\ref{r:bisbounds}). 
First, Newton iterations require more sophisticated oracles
which compute the value of the game and a left optimal strategy. 
Second, we have only used
the integrality of the method in Proposition~\ref{termination-newt}, 
so the bound on the number of iterations is rough. However, 
the positive Newton method is an interesting alternative to 
the bisection method, since it preserves feasibility. 
Therefore, it may be more sensitive to the geometry of the feasible set, 
which is especially convenient if this set  
has only few generators or its dimension is small. 
The experiments of Subsection~\ref{ss:NumExp} 
indicate that this is indeed the case, 
and the worst-case complexity bound of Theorem~\ref{posnewt-comp}  
(using Proposition~\ref{termination-newt}) is often too pessimistic. 
The main reason to give the result of Theorem~\ref{posnewt-comp} 
is that it shows the method is pseudo-polynomial.
  
Not aiming to obtain a better overall worst-case complexity result, 
in the next subsection we will rather consider the implementation of
the positive Newton method,
reducing the computation of $\lambda_k$ to a (polynomial-time solvable) shortest path problem.  
Subsection~\ref{ss:left-optim} will be devoted to the 
computation of left optimal strategies in the integer case by means
of perturbed mean payoff games. As noticed in 
Proposition~\ref{termination-newt}, 
Newton iterations should work also in the case of real coefficients.
For this we propose the algebraic approach of Subsection~\ref{ss:germs}, 
encoding a perturbed game as a game over the semiring of germs.

\subsection{Newton iterations by means of Kleene star}\label{ss:kleene}

In this subsection we show that in the case of
Algorithm~\ref{a:pos-newton} the steps of Newton method 
can be performed by calculating least solutions of inequalities of the form
$z \geq E z \vee h$, as in Proposition~\ref{bellman}.

Assume that we are at iteration $k$ of Algorithm~\ref{a:pos-newton}, 
so that we need to compute 
$\lambda_k=\min\{\lambda\in\Rmax\mid \spectral^{\sigma}(\lambda)\geq 0\}$, 
where $\sigma$ is a left optimal strategy for player Max at $\lambda_{k-1}$. 
If we set $V^{\sigma}(\lambda)$ instead of $V(\lambda)$ in $Uy\leq
V(\lambda)y$, by Theorem~\ref{chi-axbx} 
the minimal zero $\lambda_k$ of $\spectral^{\sigma}(\lambda)$ 
is exactly the least value of $\lambda$ for which this system 
is satisfied by some $y\in \Rmax^{n+1}$ with $y_{n+1}\neq -\infty $,  
i.e.\ we have 
\begin{equation}\label{DefLambdaK}
\lambda_k=\min\{ \lambda\in\Rmax\mid Uy\leq V^{\sigma}(\lambda)y\;,\; y_{n+1}\neq -\infty \; 
\makebox{ is solvable}\} \; .
\end{equation}
The main idea is to compute this minimal zero by 
considering the system $Uy\leq V^{\sigma}(\lambda)y$ directly. 
With this aim, we shall need the following observation.  

\begin{lemma}\label{problems-equiv2}
Assume that at iteration $k$ of Algorithm~\ref{a:pos-newton} 
we have $l:=\sigma(m+1)\neq n+1$, 
where $\sigma$ is a left optimal strategy for player Max at $\lambda_{k-1}$. 
Then, if   
\begin{equation}\label{ConditionEq}
\left( Uy\leq V^{\sigma}(\lambda)y\makebox{ and } y_{n+1}\neq -\infty \right) 
\implies y_l\neq -\infty \; ,
\end{equation} 
for all $\lambda$, it follows that   
\begin{equation}\label{EqDefLambdaK}
\lambda_k=\min\{ \lambda\in\Rmax\mid Uy\leq V^{\sigma}(\lambda)y\;,\; y_{l}\neq -\infty \; 
\makebox{ is solvable}\} \; .
\end{equation}
Otherwise, i.e.\ if Condition~\eqref{ConditionEq} does not hold, 
$\lambda_k=-\infty$. 
\end{lemma}

\begin{proof}
Condition~\eqref{ConditionEq} implies that the minimum 
in~\eqref{EqDefLambdaK} is less than or equal to that in~\eqref{DefLambdaK}.

To show the converse, suppose that for some $\lambda\in \R$  
there exists $\hat{y}\in \Rmax^{n+1}$ such that 
$\hat{y}_l\neq -\infty$ and $U\hat{y}\leq V^{\sigma}(\lambda)\hat{y}$. 
Since $\spectral^{\sigma}(\lambda_{k-1})=\spectral(\lambda_{k-1})\geq 0$, 
by Theorem~\ref{chi-axbx} there exists a solution $\tilde{y}$ of 
$Uy\leq V^{\sigma}(\lambda_{k-1})y$  
(and so, in particular, of the first $m$ inequalities of this system, 
i.e.\ $C y\leq D^{\sigma} y$)   
such that $\tilde{y}_{n+1}\neq -\infty$. Then, 
for any $\beta \in \R$ the combination 
$\overline{y}=\hat{y}\vee \Tilde{y}-\beta$ satisfies 
the first $m$ inequalities in $Uy\leq V^{\sigma} (\lambda)y$ 
(in other words, we have $C \overline{y}\leq D^{\sigma} \overline{y}$) 
as a tropical linear combination of 
solutions of this system of tropically linear inequalities. 
Moreover, if $\beta$ is sufficiently large, 
$\overline{y}$ also satisfies the last inequality
$uy\leq \lambda+v^{\sigma}y$ of the system 
$Uy\leq V^{\sigma} (\lambda)y$ because 
$u\hat{y}\leq \lambda + v^{\sigma} \hat{y}$ and 
$v^{\sigma} \hat{y}=v_l+\hat{y}_l > -\infty$. 
But then $\overline{y}$ satisfies 
$U\overline{y}\leq V^{\sigma}(\lambda)\overline{y}$ 
and $\overline{y}_{n+1}\geq \tilde{y}_{n+1}-\beta \neq -\infty$.
This shows that the minimum in~\eqref{DefLambdaK} is less than or
equal to that in~\eqref{EqDefLambdaK}.

Finally, if Condition~\eqref{ConditionEq} does not hold, 
for some $\bar{\lambda}$ there exists a solution $\bar{y}$ of the system
$Uy\leq V^{\sigma}(\bar{\lambda})y$ such that  
$\bar{y}_{n+1}\neq -\infty$ but $\bar{y}_l= -\infty$. Since 
$u\bar{y}\leq \bar{\lambda }+ v^{\sigma} \bar{y}=\bar{\lambda }+v_l+\bar{y}_l$, 
this can only happen if $u\bar{y}=-\infty$, 
which implies that $\bar{y}$ satisfies $U\bar{y}\leq V^{\sigma}(\lambda)\bar{y}$ 
for any $\lambda \in \R$, and so $\lambda_k=-\infty$. 
\end{proof}

We next show how to make sure that Condition~\eqref{ConditionEq} is
satisfied. Note that this condition is not satisfied if, and only if, 
for some $\lambda $ the system $Uy\leq V^{\sigma}(\lambda)y$ 
has a solution $\bar{y}$ with 
$\bar{y}_{n+1}\neq -\infty$ but $\bar{y}_l=-\infty$. The latter 
implies the existence of a solution $\bar{y}$ of 
$C y\leq D y$ such that $\bar{y}_i=-\infty$ for all 
$i\in\supp(u)$, but $\bar{y}_{n+1}\neq -\infty$ 
(so in particular Condition~\eqref{ConditionEq} is
satisfied if $n+1\in \supp(u)$). Eliminating from the system 
$C y\leq D y$ the columns corresponding to the indices in $\supp(u)$, 
the existence of such a solution is reduced to the solvability of a 
two-sided homogeneous system with the condition $y_{n+1}\neq -\infty$, 
which can be decided using a mean payoff game oracle. 
If this problem has no solution, 
then Condition~\eqref{ConditionEq} is satisfied. Otherwise, 
the value of the original tropical linear-fractional programming problem is $-\infty$. 
 
As a consequence of the previous discussion, in what follows we assume
that it has already been checked that Condition~\eqref{ConditionEq} is satisfied, 
and we explain how to perform Newton iterations in that case.  

Suppose that we are at iteration $k$ of Algorithm~\ref{a:pos-newton}, 
and let $\sigma$ be a left optimal strategy at $\lambda_{k-1}$. 
Then, if we set $l:=\sigma(m+1)$, by Lemma~\ref{problems-equiv2} we have   
\begin{eqnarray*}
\lambda_k = \min\{ \lambda \in\Rmax \mid Uy\leq V^{\sigma}(\lambda)y\;,\; y_l\neq -\infty \; 
\makebox{ is solvable}\} \; .
\end{eqnarray*} 
Since the system $Uy\leq V^{\sigma}(\lambda)y$ 
is satisfied by some $\bar{y}$ with $\bar{y}_l\neq -\infty$ if, 
and only if, it is satisfied by some $\hat{y}$ with $\hat{y}_l=0$ 
(it is enough to define $\hat{y}_i=\bar{y}_i-\bar{y}_l$ for all $i\in [n+1]$), 
it follows that 
\begin{eqnarray*}
\lambda_k&=&\min\{ \lambda \in\Rmax \mid Uy\leq V^{\sigma}(\lambda)y\;,\; y_l= 0 \; 
\makebox{ is solvable}\} \\ 
&=& \min \{ \lambda \in\Rmax \mid Cy\leq D^{\sigma}y \; ,\; 
uy\leq \lambda + v_l + y_l \; , \; y_l=0 \; \makebox{ is solvable}\} 
\; .
\end{eqnarray*}
Thus, setting $y_l=0$  
we are in the situation of problem~\eqref{problem-straight},  
because $\lambda_k+v_l$ is given by:
\begin{equation}\label{problem-sigma}
\begin{split}
&\makebox{minimize }\quad  p x\vee r \\
&\makebox{subject to:}\quad Ax\vee c \leq B^{\sigma}x\vee d^{\sigma}
  \; , \; x\in\Rmax^n 
\end{split}
\end{equation}
where $p\in\Rmax^n$, $r\in\Rmax$, $c,d^{\sigma}\in\Rmax^m$
and $A,B^{\sigma}\in\Rmax^{m\times n}$ are such that
\[
U=
\begin{pmatrix}
A & c\\
p & r
\end{pmatrix}\quad \makebox{and}\quad
V^{\sigma}(\lambda)=
\begin{pmatrix}
B^{\sigma} & d^{\sigma}\\
-\infty  &  \lambda+v_l
\end{pmatrix} \enspace .
\]
Here, just for the simplicity of the presentation, 
column $l$ is in the place of column $n+1$ when $l\neq n+1$ 
(in other words, the columns of $A$ and $B^{\sigma}$ are respectively 
the columns of $C$ and $D^{\sigma}$ with exception of column $l$, 
$c$ is column $l$ of $C$, $d^{\sigma}$ is column $l$ of  $D^{\sigma}$, 
$r=u_l$ and $p_i=u_i$ for $i\neq l$).

We claim that the system of constraints (the second line) 
in~\eqref{problem-sigma} has a least solution, 
which then minimizes $px\vee r$, and we explain how to find it.
First note that the constraints in~\eqref{problem-sigma} 
can be written as:
\begin{equation}\label{two-subsystems}
\begin{split}
(Ax)_i\vee c_i&\leq b_{i\sigma(i)}+x_{\sigma(i)}\; ,\quad 
\makebox{if} \; \sigma (i) \neq l \;  ,\\
(Ax)_i\vee c_i&\leq d_i\enspace , \quad \makebox{if} \; \sigma (i)= l\; .
\end{split}
\end{equation}

In order to find the least solution of this system, 
observe that the second subsystem can be dispensed with. 
Indeed, since 
$\spectral^{\sigma}(\lambda_{k-1})=\spectral(\lambda_{k-1})\geq 0$, 
by Theorem~\ref{chi-axbx} there exists a solution $\tilde{y}$ of 
$Uy\leq V^{\sigma}(\lambda_{k-1})y$    
such that $\tilde{y}_{n+1}\neq -\infty$, 
and so this solution also satisfies $\tilde{y}_{l}\neq -\infty$ 
(recall we assume that 
Condition~\eqref{ConditionEq} holds). Then, 
if $\tilde{x}\in \Rmax^n$ is the vector defined by
$\tilde{x}_i:=\tilde{y}_i-\tilde{y}_l$ for $i\neq l$, 
it follows that $\tilde{x}$ is a solution of~\eqref{two-subsystems}. 
Hence, if the first subsystem has the least solution $\underline{x}$, 
we have $\underline{x}\leq \tilde{x}$ and so $\underline{x}$ 
is also a solution of the second subsystem.

To show that the first subsystem in~\eqref{two-subsystems} has 
a least solution, first note that a system of two inequalities of the form 
\[
\begin{split}
r_{11}+x_1\vee \cdots \vee r_{1n}+x_n&\leq x_1 \\
r_{21}+x_1\vee \cdots \vee r_{2n}+x_n&\leq x_1
\end{split}
\]
is equivalent to just one inequality:
\[
s_1+x_1\vee \cdots \vee s_n+x_n\leq x_1 \; ,
\]
where $s_i=r_{1i}\vee r_{2i}$ for $i=1,\ldots, n$. 
Using this kind of reduction, the first subsystem can be
transformed in no more than $m(n+1)$ operations to an equivalent
system of the form
\begin{equation}\label{eq-sys}
Ex_I\vee Fx_J\vee h\leq x_I\enspace ,
\end{equation}
where $x_I$ is the sub-vector whose coordinates appear
on the right-hand side of the first subsystem in~\eqref{two-subsystems},
and $x_J$ is the sub-vector corresponding to the rest of the 
coordinates which are present in that system. 
Since we are interested in the least solution, 
we can set $x_J\equiv -\infty $, 
and then the remaining system is just of the form
\begin{equation}\label{kleene-ineq}
E z\vee h\leq z\enspace ,
\end{equation}
where $z=x_I$. By Proposition~\ref{bellman},
the least solution to this system in $\RRbar^{|I|}$ is given by
\[
\underline{z}=E^*h=h\vee Eh\vee E^2h\vee E^3h\vee\cdots
\]
As $\tilde{x}_I$ satisfies~\eqref{kleene-ineq}, 
we have $\underline{z}\leq \tilde{x}_I$ and so $\underline{z}\in\Rmax^{|I|}$.

Thus, we have the following method:

\begin{myalgorithm}\label{a:problem-sigma}
Solving~\eqref{problem-sigma}

{\bf Step 1.} Split the system $Ax\vee c\leq B^{\sigma}x\vee d^{\sigma}$
in two subsystems as in~\eqref{two-subsystems} 
and transform the first subsystem to the form~\eqref{eq-sys}. 

{\bf Step 2.} Compute $\underline{z}=E^*h$.
Set $\underline{x}_I=\underline{z}$ and $\underline{x}_J\equiv -\infty$. 

{\bf Step 3.} Return $p\underline{x}\vee r$. 
\end{myalgorithm}

We also conclude the following. 

\begin{proposition}\label{prop-newnice2}
The problems 
$\min\{\lambda \mid \spectral^{\sigma}(\lambda)\geq 0\}$
can be solved in $O(mn)+O(n^3)$ time.  
\end{proposition}

Note that in general, a system of the form~\eqref{two-subsystems} 
is solvable in $\Rmax^n$ if, and only if, 
the least solution of the first subsystem belongs to $\Rmax^n$
and satisfies the second subsystem. 

\begin{example}\label{Example3}
Consider the following tropical linear-fractional programming problem: 
\[
\begin{split}
&\text{minimize } \quad \lambda  \\
&\text{subject to:}\quad   u y\leq \lambda + v y\; ,
\; C y\leq D y\; , \; y_{4}\neq -\infty \; , \; y\in\Rmax^{4} \; , 
\; \lambda \in\Rmax 
\end{split}
\]
where 
\[
C=
\left(\begin{array}{cccc}
-3      & -4      & -\infty & -\infty  \\
-1      & -\infty & -\infty & 1 \\
-\infty & -\infty & -\infty & 0\\
1       & -\infty & 0       & -\infty 
\end{array}\right) \; ,\quad 
D=
\left(\begin{array}{cccc}
-\infty & -\infty & -\infty & 0  \\
-\infty & 0       & -\infty & -\infty \\
0       & -\infty & -\infty & -\infty \\
0       & -\infty & -\infty  & 3 
\end{array}\right) \; ,
\]
$u=(-\infty,0,-\infty,-\infty)$ and 
$v=(3,-\infty,-\infty,-\infty)$, 
so in this case we have $m=4$ and $n=3$ 
with the notation of Problem~\eqref{problem-cones}.  

Before performing Newton iterations, we need to 
check whether the system $Cy\leq Dy$ has solutions with
$y_2=-\infty$ but $y_4\neq -\infty$ (since $\supp(u)=\{2\}$). 
By the second inequality of this system, it follows
that this is impossible. Hence, Condition~\eqref{ConditionEq} is satisfied 
and so we can use~\eqref{EqDefLambdaK} in order to compute $\lambda_k$. 
Moreover, in this example Newton method requires no more than two iterations, 
since player Max has only two strategies, 
which correspond to the two finite entries in the last row of $D$.

Assume that we start with $\lambda_0=0$. Then, 
an optimal strategy $\sigma$ for player Max at $\lambda_0$ is given by: 
$\sigma(1)= 4$, $\sigma(2)= 2$, $\sigma(3)= 1$, $\sigma(4)= 4$ and 
$\sigma(5)= 1$. 
Since $l=\sigma(m+1)=\sigma(5)=1$, we have 
\[
\lambda_1=\min\{ \lambda \in\Rmax \mid Cy\leq D^{\sigma}y\;,\; y_2=uy\leq \lambda+v_1+y_1\; ,\; y_1= 0 \; 
\makebox{ is solvable}\} \; , 
\]
and so, setting $y_1=0$, $\lambda_1+v_1=\lambda_1+3$ is given by:
\begin{equation}\label{problem-sigma-Example}
\begin{split}
&\makebox{minimize }\quad y_2 \\
&\makebox{subject to:}\; (y_2-4)\vee (-3) \leq y_4\; , \; 
(y_4+1)\vee (-1) \leq y_2 \; ,\; 
y_4\leq 0 \; ,\; 
y_3 \vee 1 \leq y_4+3 \;   .
\end{split}
\end{equation} 
Note that the system of constraints in~\eqref{problem-sigma-Example},  
obtained by setting $y_1=0$ in $C y\leq D^{\sigma} y$ 
(i.e., in~\eqref{two-subsystems} the first column plays the role of 
free term), reduces to 
\[
E \left(\begin{array}{cc}
y_2  \\
y_4  
\end{array}\right) 
\vee F y_3\vee h\leq 
\left(\begin{array}{cc}
y_2  \\
y_4  
\end{array}\right) 
\; ,\; 
y_4\leq 0 \; ,
\]
where
\[
E=
\left(\begin{array}{cc}
-\infty & 1  \\
-4       & -\infty  
\end{array}\right) \; , \;
F=
\left(\begin{array}{cc}
-\infty  \\
-3 
\end{array}\right) \; , \;
h=
\left(\begin{array}{cc}
-1  \\
-2  
\end{array}\right) \; .
\]
Since 
\[
E^* h=
\left(\begin{array}{cc}
\;0 & \;1\;  \\
-4 & \;0\;  
\end{array}\right) 
\left(\begin{array}{cc}
-1  \\
-2  
\end{array}\right) =
\left(\begin{array}{cc}
-1  \\
-2  
\end{array}\right) \; ,
\]
the least solution of the system of 
constraints in~\eqref{problem-sigma-Example} 
is $(y_2,y_3,y_4)^T=(-1,-\infty,-2)$. Therefore, 
the value of problem~\eqref{problem-sigma-Example} is 
$-1$, and thus $\lambda_1=-4$. 
It can be checked that this is the optimal solution of 
the tropical programming problem (and in particular, 
that $\sigma$ is still a left optimal strategy for player Max at 
$\lambda_1=-4$).    
\end{example} 

\subsection{Computing left optimal strategies}\label{ss:left-optim}

In order to compute left optimal strategies, 
consider the mean payoff game associated with 
the tropical linear-fractional programming problem~\eqref{problem-cones}, 
and let the weights $\lambda+v$ of the arcs connecting node $m+1$ of Max  
with nodes of Min be replaced by $\lambda-\epsilon+v$, where
$\epsilon\in \R$. 
In this way, we obtain a {\em perturbed mean payoff game}. 
For small enough $\epsilon>0$, the optimal strategies for this game are 
the left optimal strategies required by the positive Newton iterations.
Here, we will require the mean payoff oracle to find optimal strategies, 
not just the value of the game. 
The complexity of such oracle will be denoted by $\MPGstrong(m,n,M)$, 
for mean payoff games with integer payments whose absolute values 
are bounded by $M$,
with $m$ nodes of Max and $n$ nodes of Min. 
A pseudo-polynomial algorithm for computing 
optimal strategies of such games is described in~\cite{ZP-96}.

\begin{proposition}\label{p:germs3}
If the finite coefficients in~\eqref{problem-cones} are integers 
with absolute values bounded by $M$, then at each iteration of the positive 
Newton method, started and finished as in Remark~\ref{r:newtstart}, 
a left optimal strategy can be found in 
\[
\MPGstrong(m+1,n+1,(\min(m,n)+2)(M+2M(\min(m,n)+1))+1)
\]
operations. 
\end{proposition}

\begin{proof}
By Corollary~\ref{c:philambda}, it follows that each $\lambda_k$ is integer 
and the breaking points of $\spectral(\lambda)$ are rational numbers 
whose denominators do not exceed $\min(m,n)+1$.  
Therefore, optimal strategies for the game at
$\lambda^*=\lambda_k-1/(\min(m,n)+2)$ are left optimal strategies at 
$\lambda_k$. Then, we only need to apply a mean payoff oracle in order to 
compute optimal strategies for the perturbed mean payoff game 
with $\epsilon=1/(\min(m,n)+2)$. Multiplying (in the usual sense) 
the payments by $\min(m,n)+2$ 
we obtain a mean payoff game with integer payments and with the same 
optimal strategies (in this sense, equivalent to the perturbed game). 
The computation of optimal strategies in the latter game takes no more than 
$\MPGstrong(m+1,n+1,(\min(m,n)+2)(M+2M(\min(m,n)+1))+1)$ operations, 
because the payments in the perturbed mean payoff game are either of the 
form $a$ or $a+\lambda_k -\epsilon $, 
where $|a|\leq M$ and $|\lambda_k| \leq 2M(\min(m,n)+1)$, 
and these payments are multiplied by $\min(m,n)+2$.
\end{proof}

\begin{example}
In order to find a left optimal strategy $\sigma$ 
for player Max at $\lambda_{k-1}=0$ in Example~\ref{Example3}, 
we only need to compute an optimal strategy for the associated game at 
$\lambda^*=\lambda_{k-1}-1/(\min(m,n)+2)=-1/5$. This can be done 
by solving the game whose payments are given by the matrices 
\[
\left(\begin{array}{cccc}
-15      & -20      & -\infty & -\infty  \\
-5      & -\infty & -\infty & 5 \\
-\infty & -\infty & -\infty & 0\\
5       & -\infty & 0       & -\infty \\
-\infty & 0 & -\infty & -\infty  
\end{array}\right) \; \makebox { and } \;
\left(\begin{array}{cccc}
-\infty & -\infty & -\infty & 0  \\
-\infty & 0       & -\infty & -\infty \\
0       & -\infty & -\infty & -\infty \\
0       & -\infty & -\infty  & 15 \\
14 & -\infty  & -\infty & -\infty  
\end{array}\right) \; ,
\]
which are obtained by multiplying (in the usual sense) 
the payments for the game at $\lambda^*$ by $5$. 
\end{example}

\begin{remark}
\label{r:noleftopt}
Proposition~\ref{p:germs3} was necessary to establish the pseudo-polynomiality
of the positive Newton method, which regularly uses left optimal strategies.
However, as observed in Remark~\ref{r:justopt}, 
the use of left optimal strategies is not necessary when 
$\spectral(\lambda_k)>0$. Moreover, 
when $\spectral(\lambda_k)=0$ and the coefficients are integers, 
an alternative to computing a left optimal strategy is to use an
optimal strategy, checking whether $\spectral(\lambda_k-1)<0$ when 
$\lambda_{k+1}=\lambda_k$. In that case, Corollary~\ref{c:philambda} 
part~\eqref{c:philambdaP2} guarantees that $\lambda_k$ is optimal. Otherwise,
proceed with $\lambda_{k+1}:=\lambda_k-1$. With this modification, 
the complexity of the computation of optimal strategies falls to 
\[
\MPGstrong(m+1,n+1,M+2M(\min(m,n)+1)) \; ,
\]
instead of the bound of Proposition~\ref{p:germs3}. 

\end{remark}

We are now ready to sum up the computational 
complexity of the positive Newton method with left optimal strategies.

\begin{theorem}\label{posnewt-comp}
If the finite coefficients in~\eqref{problem-cones} are integers 
with absolute values bounded by $M$, then the positive Newton method, 
started and finished as in Remark~\ref{r:newtstart}, takes no more than
\[
\begin{split}
O(M &\min(m,n)) \times (O(mn)+O(n^3)+\\
& +\MPGstrong(m+1,n+1,(\min(m,n)+2)(M+2M(\min(m,n)+1))+1))
\end{split}
\] 
operations. In particular, the positive Newton method is pseudo-polynomial.
\end{theorem}

\begin{proof}
As the numbers $\lambda_k$ generated by the 
positive Newton method are integers and lie within
$[-2M(\min(m,n)+1),2M(\min(m,n)+1)]$, 
the number of iterations does not exceed $4M(\min(m,n)+1)+1$.
Combining this with Propositions~\ref{prop-newnice2} and~\ref{p:germs3}, 
we get the claim. 
\end{proof}

Note that Remark~\ref{r:noleftopt} can be used to reduce the bound of 
Theorem~\ref{posnewt-comp}, if the left optimal strategies are not used,
getting rid of the factor $(\min(m,n)+2)$ in the third argument of 
$\MPGstrong$.

\subsection{Perturbed mean payoff games as mean payoff games over germs}\label{ss:germs}

Next we discuss an alternative to the perturbation technique
of the previous subsection: instead of considering the mean payoff
game for several values of the perturbation parameter, we may consider
a mean payoff game the payments of which belong to a lattice ordered
group of {\em germs}. In a nutshell, the elements of this group
encode infinitesimal perturbations of the payments.
This algebraic structure
allows one to deal more generally with one-parameter perturbed games
(not only the ones arising from tropical linear-fractional programming).
A similar structure appeared in~\cite{gg0}.
This is somehow analogous to the perturbation methods 
used to avoid degeneracy in linear programming.
We hope to develop this further
in a subsequent work. The materials of this subsection are
not used in the rest of the paper. However, this alternative can be useful in two respects:
1) to develop the present Newton method in the case of real coefficients, 2) to improve the complexity result of the
previous subsection. 

Consider a {\em mean payoff game over germs}, finite-duration version, where the weights of arcs in 
$\Bipdig$ are pairs of real numbers $(a,b)$ endowed with lexicographic order:
\begin{equation}
\label{lex}
(a^1,b^1)\leqlex (a^2,b^2)\Leftrightarrow 
\begin{cases}
a^1<a^2,\ \text{or}\\
a^1=a^2,\ b^1\leq b^2,
\end{cases}
\end{equation}
and the componentwise addition is used to calculate the weights of paths 
(or cycles). These games correspond to two-sided tropical linear systems 
over the semiring of germs $\Germ:=\R^2\cup\{(-\infty,-\infty)\}$, 
where for $g^1=(a^1,b^1), g^2=(a^2,b^2)\in \Germ$ we define 
$\ltr g^1+g^2\rtr:=\max(g^1,g^2)$ following~\eqref{lex} 
and $\ltr g^1g^2\rtr:=(a^1+a^2,b^1+b^2)$. 

With a game over germs we associate an 
{\em $\epsilon$-perturbed mean payoff game}, for $\epsilon \geq 0$,
in which the weights $(a,b)$ of the arcs are replaced by $a+\epsilon b$. 
If the payments in a mean payoff game over germs 
are given by the matrices $A$ and $B$ (with entries in $\Germ$), 
then the matrices associated with the corresponding 
$\epsilon$-perturbed mean payoff game will be denoted by 
$A(\epsilon)$ and $B(\epsilon)$, respectively. 

\begin{proposition}\label{p:germs1}
Suppose that the matrices of payments in a mean payoff game over germs
(finite-duration version) satisfy Assumptions 1 and 2. 
Then this game has a value and positional optimal 
strategies (meaning that \eqref{val-exist} holds for mean payoff games over germs). 
Moreover, if $(\chi_i,\kappa_i)$ is the value of such a game, 
then there exists $\overline{\epsilon}>0$ such that for any 
$0<\epsilon<\overline{\epsilon}$ the associated $\epsilon$-perturbed
mean payoff game has value $\chi_i+\epsilon\kappa_i$ 
and these games have common positional optimal strategies. 
\end{proposition}

\begin{proof}
Note that if the matrices of payments $A$ and $B$ in a mean 
payoff game over germs satisfy Assumptions 1 and 2, 
then for any $\epsilon$-perturbed mean payoff game the corresponding 
matrices $A(\epsilon)$ and $B(\epsilon)$ also satisfy these assumptions.
 
Let $\delta$ be the minimal absolute value of nonzero differences 
between cycle means in the mean payoff game with payments given by 
$A(0)$ and $B(0)$, and let $M$ be the greatest absolute value of the second
component of germs. Define $\overline{\epsilon}:=\delta/4M$ 
and consider any $\epsilon $ such that $0<\epsilon<\overline{\epsilon}$. 

By~\eqref{val-exist}, for the $\epsilon$-perturbed mean payoff game 
there exist positional strategies $\sigma^*$ and $\tau^*$ such that 
\[ 
\Phi_{A(\epsilon),B(\epsilon)}(j,\tau^*,\sigma)\leq 
\Phi_{A(\epsilon),B(\epsilon)}(j,\tau^*,\sigma^*)\leq 
\Phi_{A(\epsilon),B(\epsilon)}(j,\tau,\sigma^*)
\] 
for all (not necessarily positional) strategies $\sigma$ and $\tau$. 
Let $\sigma $ be any strategy for player Max and assume that 
$\Phi_{A(\epsilon),B(\epsilon)}(j,\tau^*,\sigma)=a +\epsilon b$ and 
$\Phi_{A(\epsilon),B(\epsilon)}(j,\tau^*,\sigma^*)=\chi_j+\epsilon \kappa_j$. 
If $a=\chi_j$, we have $b\leq \kappa_j$ because 
$a +\epsilon b \leq \chi_j+\epsilon \kappa_j$. 
Otherwise (i.e., if $a\neq \chi_j$), since 
$|b|,|\kappa_j|\leq 2M$, $|a-\chi_j|\geq \delta$, 
$\epsilon <\overline{\epsilon}=\delta/4M$, and 
$a +\epsilon b \leq \chi_j+\epsilon \kappa_j$, it follows that $a<\chi_j$. 
Therefore, we conclude that 
$\Phi_{A,B}(j,\tau^*,\sigma)=(a,b)\leqlex (\chi_j,\kappa_j)=\Phi_{A,B}(j,\tau^*,\sigma^*)$. 
The same argument shows that 
$\Phi_{A,B}(j,\tau^*,\sigma^*)\leqlex \Phi_{A,B}(j,\tau ,\sigma^*)$
for any strategy $\tau$ for player Min. 
This proves that the finite duration version of the mean 
payoff game over germs  has a value, 
given by $(\chi_j,\kappa_j)$, 
and that positional optimal strategies for 
the $\epsilon$-perturbed mean payoff game 
(with $\epsilon <\overline{\epsilon}$)
are also optimal for the game over germs.  

Assume now that $\sigma^*$ and $\tau^*$ are positional optimal strategies 
for the finite duration version of the mean payoff game over germs, 
and let $(\chi_j,\kappa_j)$ be its value. Then, 
$\Phi_{A,B}(j,\tau^*,\sigma)\leqlex (\chi_j,\kappa_j)\leqlex \Phi_{A,B}(j,\tau,\sigma^*)$ 
for all strategies $\sigma$ and $\tau$ (not necessarily positional). 

Let $\sigma $ be any strategy for player Max and assume that 
$\Phi_{A,B}(j,\tau^*,\sigma)=(a,b)$.  
If $a=\chi_j$, we conclude $b\leq \kappa_j$ because 
$\Phi_{A,B}(j,\tau^*,\sigma)\leqlex (\chi_j,\kappa_j)$, and thus 
$\Phi_{A(\epsilon),B(\epsilon)}(j,\tau^*,\sigma)=a+\epsilon b\leq 
\chi_j+\epsilon\kappa_j$ for any $\epsilon \geq 0$. 
Suppose now that $a<\chi_j$. 
Since $|b|,|\kappa_j|\leq 2M$ and $\chi_j-a\geq \delta$, it follows that 
$\Phi_{A(\epsilon),B(\epsilon)}(j,\tau^*,\sigma)=a+\epsilon b\leq 
\chi_j+\epsilon\kappa_j$ for any $\epsilon$ such that 
$0<\epsilon<\overline{\epsilon}$. The same argument shows that 
$\Phi_{A(\epsilon),B(\epsilon)}(j,\tau,\sigma^*)\geq \chi_j+\epsilon\kappa_j$ 
for any strategy $\tau$ for player Min  
and any $\epsilon$ such that $0<\epsilon<\overline{\epsilon}$. 
Therefore, we conclude that $\sigma^*$ and $\tau^*$ 
are positional optimal strategies 
for the $\epsilon$-perturbed mean payoff game and that 
its value is $\chi_j+\epsilon\kappa_j$. 
This proves the claim.  
\end{proof}     

\begin{remark}\label{r:germs-convenience}
Proposition~\ref{p:germs1} opens the way to using mean payoff games 
over germs in order to find right or left optimal strategies 
in the Newton methods. In that case, 
note that the second component of all finite weights must be set to $0$ 
except for the arcs connecting node $m+1$ of Max with nodes of Min, 
where it is set to $1$ (for right optimality) or to $-1$ (for left
optimality). 
This raises the issue of developing a direct combinatorial algorithm
to solve mean payoff games over germs. Such an algorithm would avoid
the perturbation technique of the previous subsection. This will be
discussed elsewhere.
\end{remark}

\begin{example}
By Proposition~\ref{p:germs1}, in Example~\ref{Example3} 
we could find a left optimal strategy $\sigma$ 
for player Max at $\lambda_{k-1}=0$  
computing an optimal strategy for the mean payoff game over germs 
whose matrices of payments are:  
\[
\left(\begin{array}{cccc}
(-3,0) & (-4,0) & \ZGerm & \ZGerm \\
(-1,0) & \ZGerm & \ZGerm & (1,0) \\
\ZGerm & \ZGerm & \ZGerm & (0,0) \\
(1,0) & \ZGerm & (0,0) & \ZGerm \\
\ZGerm & (0,0) & \ZGerm & \ZGerm  
\end{array}\right) \; \makebox { and } \;
\left(\begin{array}{cccc}
\ZGerm & \ZGerm & \ZGerm & (0,0)  \\
\ZGerm & (0,0) & \ZGerm & \ZGerm \\
(0,0) & \ZGerm & \ZGerm & \ZGerm \\
(0,0) & \ZGerm & \ZGerm  & (3,0) \\
(3,-1) & \ZGerm  & \ZGerm & \ZGerm  
\end{array}\right) \; ,
\]
where $\ZGerm:=(-\infty,-\infty)$.
\end{example}

\begin{remark}\label{r:algorithms}
Proposition~\ref{p:germs1} extends~\eqref{val-exist} to germs. 
Note that this equation provides a very crude algorithm for 
computing values and optimal strategies. 
Further idea is to allow more general algorithms, 
showing that they can be applied to mean payoff games over germs. 
This will be investigated elsewhere. 
\end{remark}

\section{Examples}

\subsection{Minimization}\label{SectionExampleMin}

\begin{figure}
\begin{center}
\begin{tabular}{c@{{}{}}c}
\begin{minipage}{0.45\textwidth}
\includegraphics[width=0.75\linewidth]{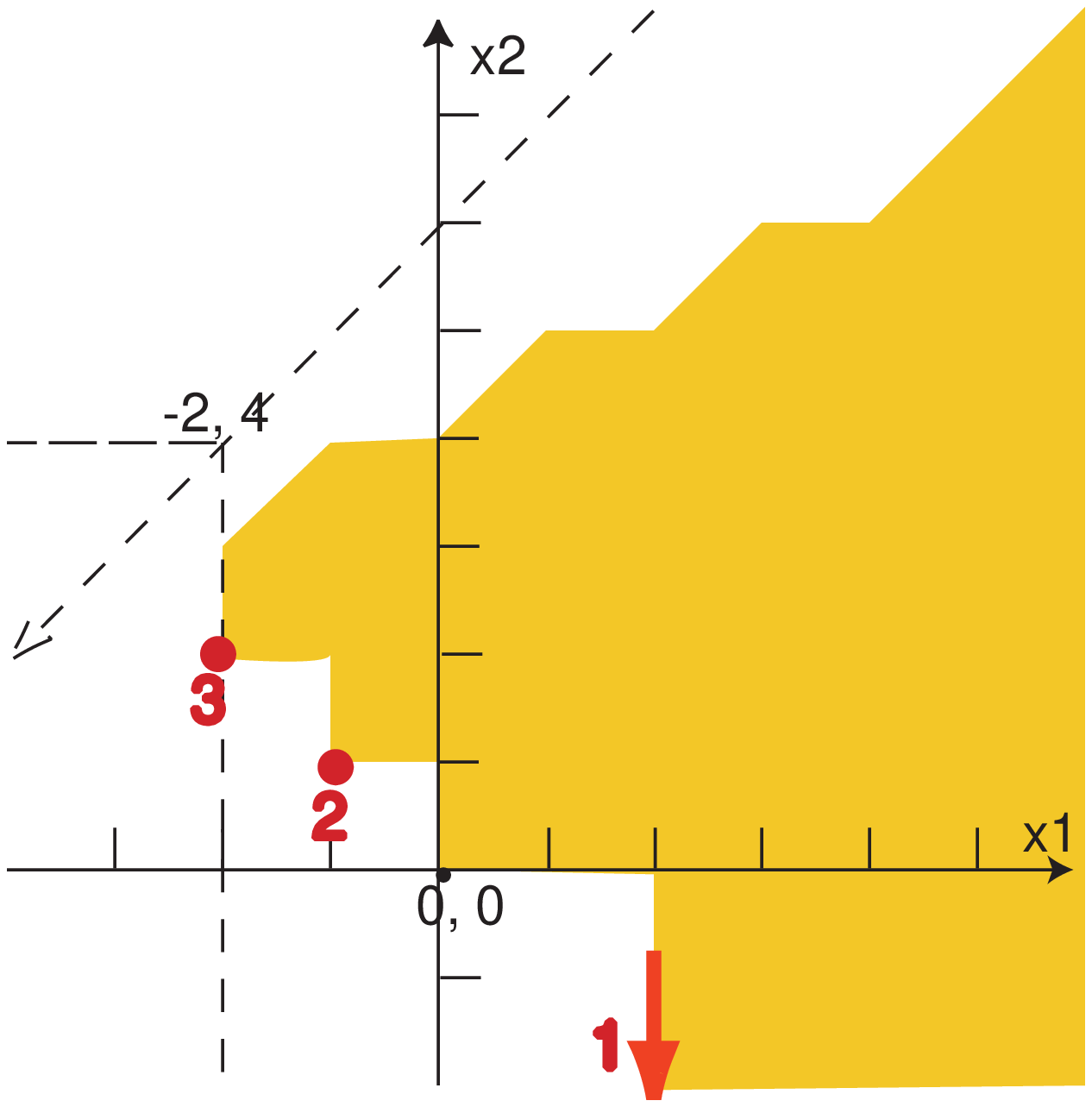} 
\end{minipage}&
\begin{minipage}{0.45\textwidth}
{{\includegraphics[width=1\linewidth]{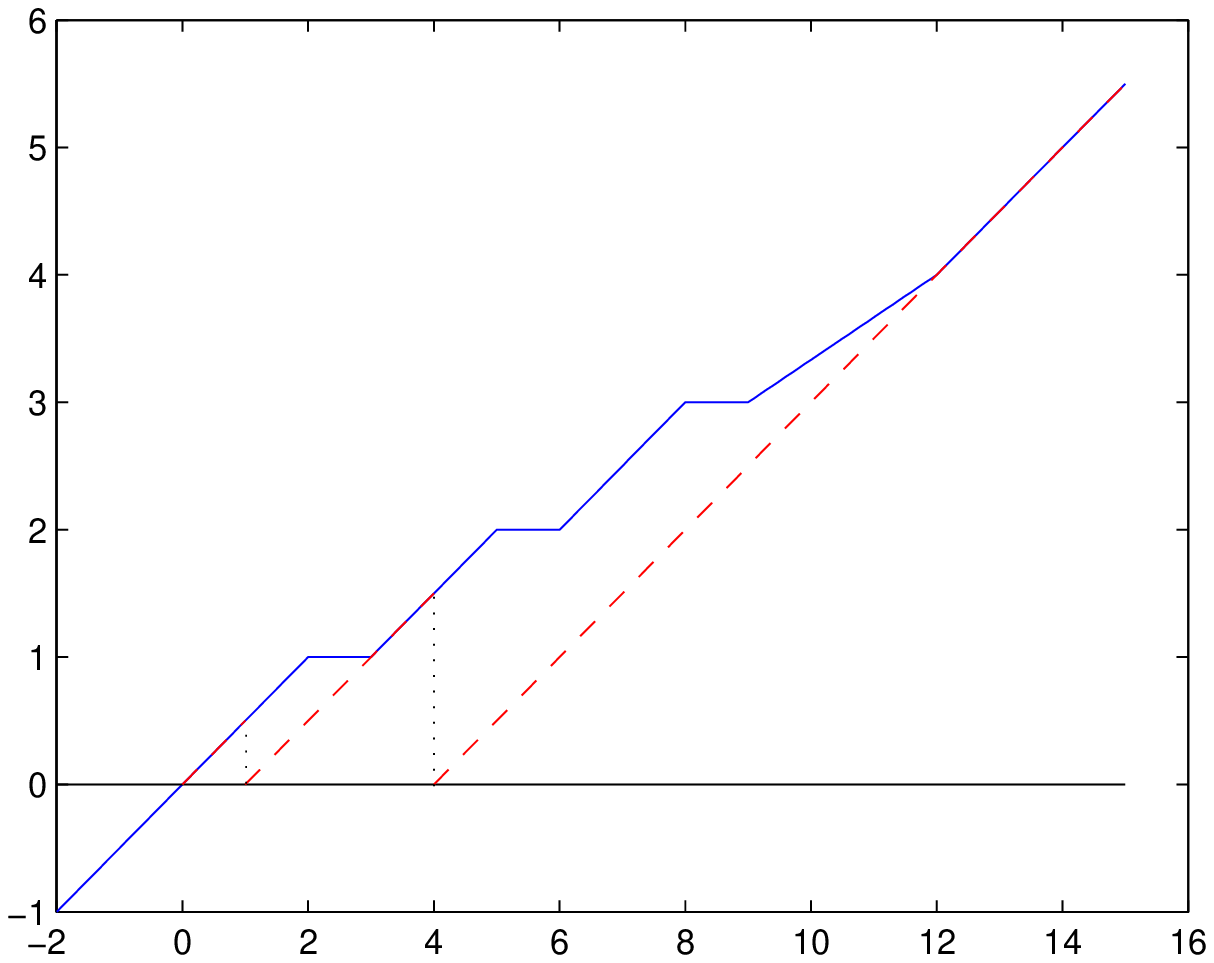}}}
\end{minipage}
\end{tabular}
\caption{The minimization problem of Example~\ref{Example2}: 
the tropical polyhedron and the spectral function.}
\label{tropprogs-ex}
\end{center}
\end{figure}

We next apply the positive Newton method to 
the minimization problem of Example~\ref{Example2}. 
Recall that the equivalent homogeneous version of this problem  
is to minimize $\lambda $ subject to 
$u y\leq \lambda + v y$, $C y \leq D y$, and $y_3\neq -\infty$, 
with $C=[A,c]$, $D=[B,d]$, $u=(2,-4,-\infty)$ and $v=(-\infty,-\infty,0)$, 
where the matrices $A,B$ and the vectors $c,d$ 
are given in Example~\ref{Example2}
(also recall that in this example, $m=7$ and $n=2$). 
Note that in this case, for any strategy $\sigma $ for player Max, 
the only possible value for $l:=\sigma(m+1)=\sigma(8)$ is $n+1=3$. 
Then, at each iteration $k$ of the positive 
Newton method applied to this problem,  
in order to compute $\lambda_k+v_l=\lambda_k$ we need to minimize 
$(2+y_1)\vee (-4+y_2)$ 
subject to the system obtained by setting 
$y_3=0$ in $C y \leq D^\sigma y$, 
as explained in Subsection~\ref{ss:kleene}. The latter is   
Problem~\eqref{problem-sigma} for this particular case.  

We start the positive Newton method with $\lambda_0=15$, 
where $\spectral(\lambda_0)=5.5$. 
The function $\sigma(1)=1$,  
$\sigma(2)=1$, \ldots , $\sigma(7)=1$ and $\sigma(8)=3$ 
is an optimal strategy for player Max at $\lambda_0$. 
To perform the first Newton iteration, 
we find the minimal solution of the system 
\[
\begin{split}
0\leq x_1-2 \; ,\; 0\leq x_1 \; , \; 0\leq 1+x_1 \; , \; x_2-4\leq x_1 \; , \\
(x_2-3)\vee 0\leq x_1+2 
\; ,\; x_2-3\leq x_1 \; ,\; x_2-2\leq x_1 \;,
\end{split}
\]
which is $(\underline{x}_1,\underline{x}_2)=(2,-\infty)$. The next value is 
$\lambda_1=(2+\underline{x}_1)\vee (-4+\underline{x}_2)=4$. 
Then, $\spectral(\lambda_1)=1.5$ 
and $\sigma(1)= 2$, $\sigma(2)= 2$, $\sigma(3)= 1$, $\sigma(4)= 1$, 
$\sigma(5)= 1$, $\sigma(6)=3$, $\sigma(7)= 3$ and $\sigma(8)=3$ 
is a new optimal strategy for player Max. 
For the next Newton iteration, we find the minimal solution of the
system
\[
\begin{split}
0\leq x_2 \; ,\; 0\leq x_2-1 \; ,\; 0\leq 1+x_1 \; ,\; 
\; x_2-4\leq x_1 \; , \\ 
(x_2-3)\vee 0\leq x_1+2 
\; ,\; x_2-5\leq 0 \; ,\; x_2-6\leq 0 \; ,
\end{split}
\]
which is $(\underline{x}_1,\underline{x}_2)=(-1, 1)$. 
Then, the next value is 
$\lambda_2=(2+\underline{x}_1)\vee (-4+\underline{x}_2)=1$. 
Now $\spectral(\lambda_2)=0.5$ 
and $\sigma(1)=2$, $\sigma(2)=2$, $\sigma(3)=2$, $\sigma(4)=1$, 
$\sigma(5)=3$, $\sigma(6)=3$, $\sigma(7)=3$ and $\sigma(8)=3$  
is the optimal strategy for player Max. For the next Newton iteration, 
we find the minimal solution of the system
\[
\begin{split}
0\leq x_2 \; ,\; 0\leq x_2-1 \; ,\; 0\leq x_2-2 \; , \; 
(x_2-3)\vee 0 \leq x_1+2 
 \; , \\ 
x_2-4\leq 0 \; ,\; x_2-5\leq 0 \; ,\; x_2-6\leq 0 \; ,
\end{split}
\]
which is $(\underline{x}_1, \underline{x}_2)=(-2, 2)$. 
This gives $\lambda_3=(2+\underline{x}_1)\vee (-4+\underline{x}_2)=0$, 
which is the optimal value $\lambda^\ast$. 
The optimality of $\lambda^\ast=0$ can be certified applying 
Theorem~\ref{1st-cert}, see Example~\ref{Example2} above. 

The vectors $(2,-\infty)$, $(-1,1)$ and $(-2,2)$
found by the Newton iterations are indicated on the left-hand side
of Figure~\ref{tropprogs-ex} as ``1'', ``2'' and ``3''.

The right-hand side of Figure~\ref{tropprogs-ex} displays the graph
of $\spectral(\lambda)$, together with the Newton iterations. The
graphs of partial spectral functions $\spectral^{\sigma}(\lambda)$
are given by red dashed lines.

\subsection{Maximization} 

\begin{figure}
\begin{center}
\begin{tabular}{c@{{}{}}c}
\begin{minipage}{0.45\textwidth}
\vspace{-1em}
\includegraphics[width=0.75\linewidth]{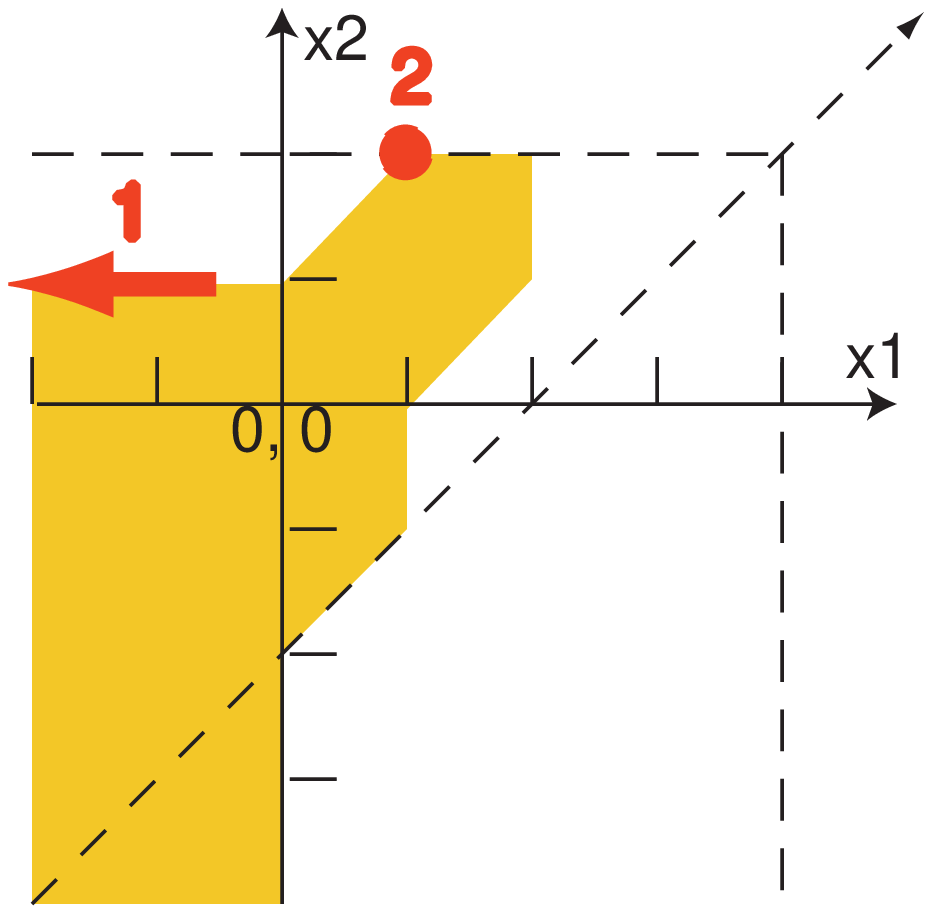} 
\end{minipage}
& 
\begin{minipage}{0.45\textwidth}
\includegraphics[width=1\linewidth]{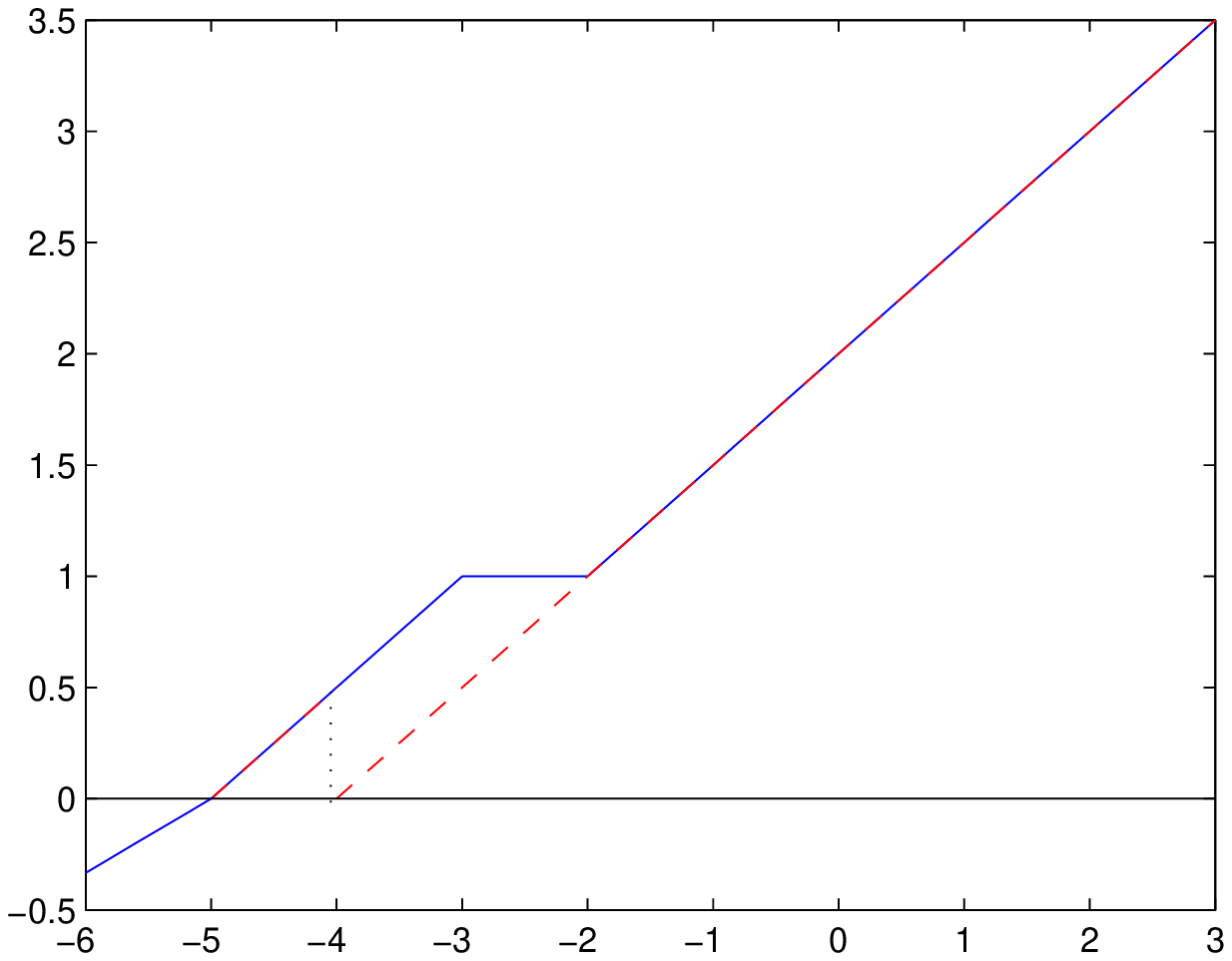}
\end{minipage}
\end{tabular}
\caption{The maximization problem of Example~\ref{Example1}: the tropical polyhedron and the spectral function.} 
\label{fig-max}
\end{center}
\end{figure}

Consider the maximization problem of Example~\ref{Example1}. 
We next apply the positive Newton method to the  
homogeneous version~\eqref{EqProbMaxExample}
of the equivalent minimization problem.

With this aim, firstly observe that in this case 
$n+1=3\in \supp (u)=\{ 3 \}$, and so Condition~\eqref{ConditionEq} 
is always satisfied. Therefore, as explained in Subsection~\ref{ss:kleene},  
at each iteration of the positive Newton method applied to this problem, 
$\lambda_k+v_l$ can be computed by minimizing $y_3$ subject to the 
system obtained by setting $y_l=0$ in $C y \leq D^{\sigma} y$, 
which corresponds to Problem~\eqref{problem-sigma} in the case of this example. 
 
Let us take $\lambda_0=3$. Then, we obtain that
$\sigma(1)= 3$, $\sigma(2)= 3$, $\sigma(3)= 3$, $\sigma(4)= 3$ and 
$\sigma(5)= 2$ is an optimal strategy for player Max at $\lambda_0$. 
To perform the Newton iteration we first notice that 
$l=\sigma(m+1)=\sigma(5)=2$, 
which means that $\lambda_1+v_2=\lambda_1+3$ is the minimum of $y_3$  
subject to the system obtained by setting 
$y_2=0$ in $C y \leq D^{\sigma} y$, 
as explained above. 
Thus, we have to find the minimal solution of the following system:
\[
\; x_3\geq -1\; , \; x_3\geq (-2+x_1)\vee (-2) 
\; ,\; x_3\geq -1+x_1\; 
,\; x_3\geq x_1\; ,
\]
which is $(\underline{x}_1,\underline{x}_3)=(-\infty,-1)$. 
The full vector $y=(-\infty , 0 ,-1)$ 
is a translate of $y+1=(x_1 , x_2 , 0)$, 
where $(x_1,x_2)=(-\infty,1)$ is marked 
as ``1'' at the left of Figure~\ref{fig-max}. 
Meanwhile we obtain $\lambda_1=\underline{x}_3-3=-4$, and
$\sigma(1)= 1$, $\sigma(2)= 3$, $\sigma(3)= 2$, $\sigma(4)= 2$ 
and $\sigma(5)= 2$ is now a new optimal strategy for player Max. 
Again, here $l=\sigma(m+1)=\sigma(5)=2$ and so the second columns of 
$C$ and $D^{\sigma}$
are the free terms in~\eqref{problem-sigma}. 
We have to find the minimal solution of the following system:
\[
\; x_1\geq -1\; ,\; x_3\geq (-2+x_1)\vee (-2) 
\; , \; 0\geq -1+x_1\; ,
\; 2\geq x_1\;,
\]
which is $(\underline{x}_1,\underline{x}_3)=(-1,-2)$. 
We obtain $\lambda_2=\underline{x}_3-3=-5$, which is the optimal value, 
and so the value of the original maximization problem is $5$. 
The full vector $y=(-1, 0, -2)$ is a
translate of $y+2=(x_1, x_2, 0)$, 
where $(x_1,x_2)=(1,2)$ 
is marked as ``2'' at the left of Figure~\ref{fig-max}.

As in the case of Figure~\ref{tropprogs-ex}, the right-hand side of
Figure~\ref{fig-max} displays the graph of $\spectral(\lambda)$,
together with the Newton iterations. The graphs of partial spectral
functions $\spectral^{\sigma}(\lambda)$ are given by red dashed lines.

\subsection{Numerical experiments}\label{ss:NumExp}
 
A preliminary implementation of the bisection and Newton methods for
tropical linear-fractional programming was developed in MATLAB. 
We next present some graphs showing how they behave on 
randomly generated instances of tropical linear-fractional programming problems,
in which the entries of matrices and vectors range from $-500$ to $500$. 
The matrices $A$ and $B$ in~\eqref{problem-straight} and~\eqref{problem-sets} 
are square, with dimensions ranging from $1$ to $400$. 

\begin{figure}
\begin{center}
\begin{tabular}{@{}c@{{}\qquad\qquad{}}c}
\begin{minipage}{0.3763158\linewidth}
\includegraphics[width=1.4\linewidth,height=0.887\linewidth]{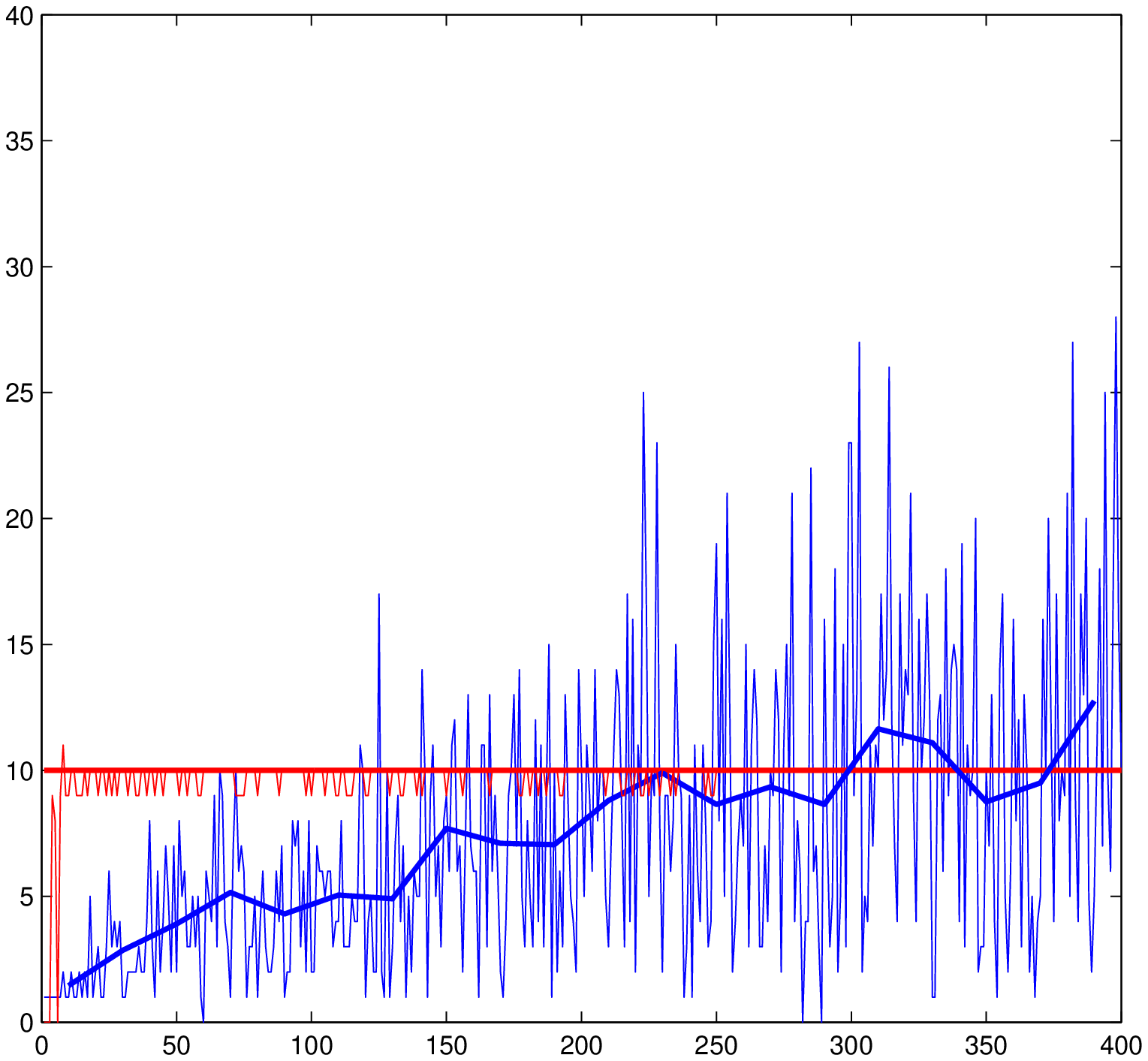}
\end{minipage} &
\begin{minipage}{0.55\linewidth}
\includegraphics[width=\linewidth]{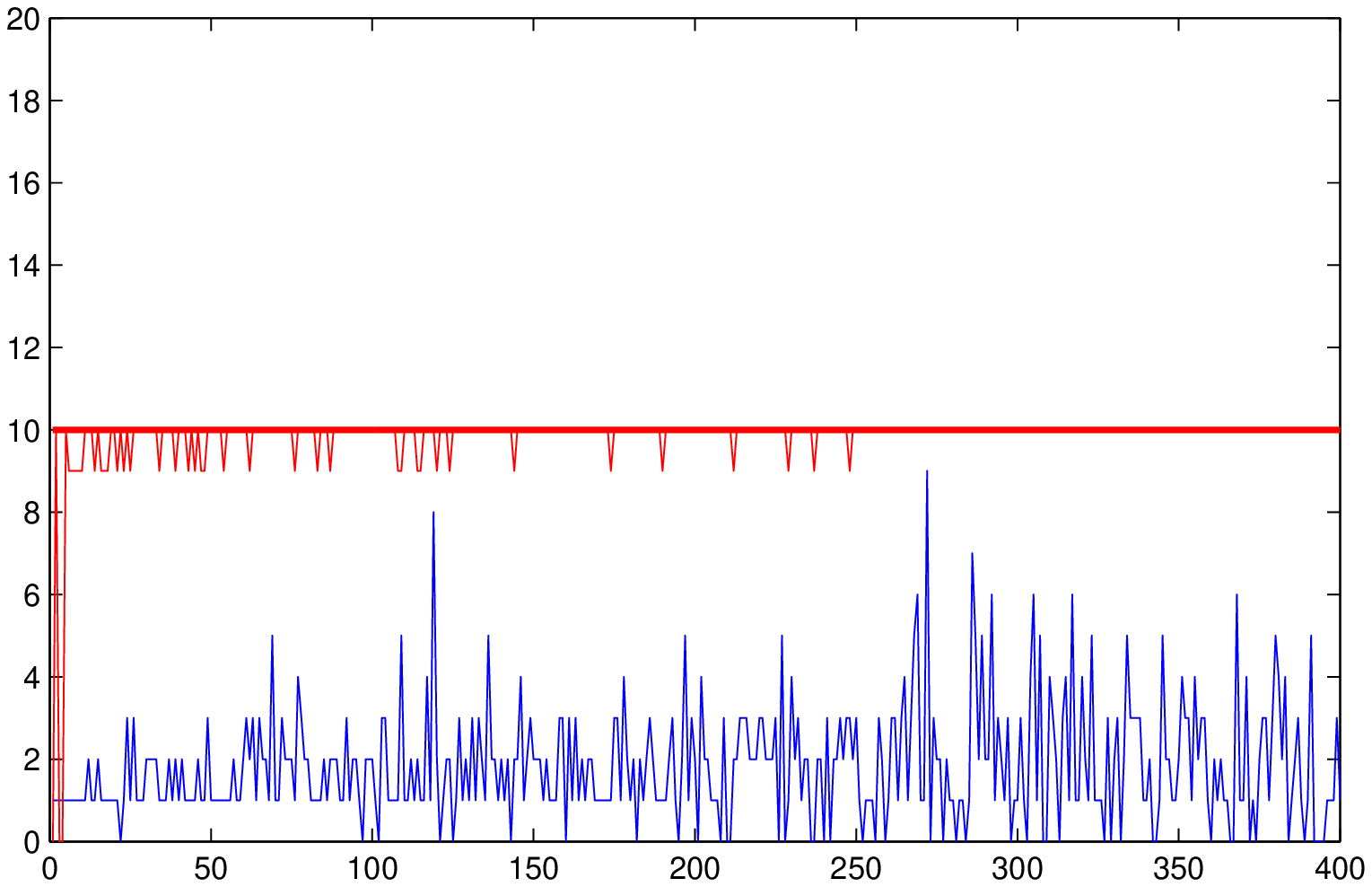}
\end{minipage}
\end{tabular}
\caption{Number of iterations of Newton method (thin blue line)
and bisection method (thin red line) in the cases of minimization (left)
and maximization (right). Thick blue line:
average number of iterations of Newton method for each
interval of 20 dimensions. Thick red line: level $\log(2M)\approx 10$.}
\label{f:minmax}
\end{center}
\end{figure}

Figure~\ref{f:minmax} displays the cases of the tropical linear 
programming~\eqref{problem-straight}, in which all the entries are finite. 
Here the certificates of unboundedness reduce to the solvability 
of a two-sided tropical system of inequalities. 
When a feasible and bounded problem is generated, 
it is solved by the bisection and Newton methods. 

For the bisection method, we use the lower initial values 
$\underline{\lambda}_0$ of~\cite{BA-08}, 
see also Remark~\ref{r:bisbounds}. Following~\cite{BA-08},
the upper initial values $\overline{\lambda}_0$ for 
the bisection method come from 
a solution of $Ax\vee c\leq Bx\vee d$. To find this solution, 
we use the policy iteration of~\cite{DG-06} 
instead of the alternating method of~\cite{CGB-03}. 
Shown by the thin red line (up to $m=n=250$), 
the bisection method worked similarly in the case of minimization and 
maximization. In our experiments, the interval between lower and upper
initial values never exceeded $2M=1000$, 
with the number of iterations quickly approaching
a constant level of $9$ or $10$ iterations ($\log M\approx 9$ or 
$\log 2M\approx 10$).   

\begin{figure}
\begin{center}
\includegraphics[width=0.5\linewidth]{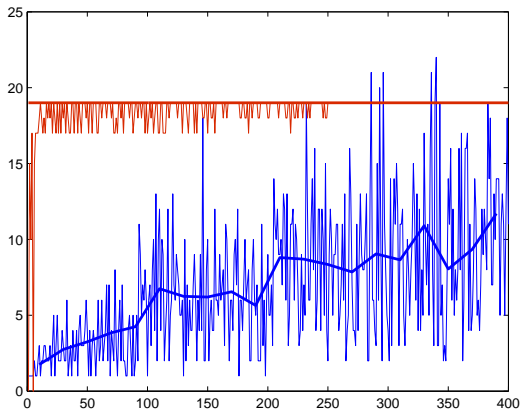}
\caption{Number of iterations of Newton method (thin blue line)
and bisection method (thin red line)
in the case of minimization with $M=500000$. Thick blue line:
average number of iterations of Newton method
for each interval of 20 dimensions. Thick red line:
level $\log(2M)-1\approx 19$.}
\label{f:bigentries}
\end{center}
\end{figure}
 
The thin blue line represents the run of Newton method, 
and the thick blue line represents their average number
calculated for each interval of $20$ dimensions. For the sake of fair comparison
with the bisection method, the initial value $\lambda_0$ coincides with
the upper initial value $\overline{\lambda}_0$ for the bisection method. 
This value comes from a solution of $Ax\vee c\leq Bx\vee d$, 
instead of the theoretical value $2M(n+1)$, 
which depends on $n$ and may be much greater.  
In the case of minimization, 
the average number of Newton iterations slowly grows, 
being smaller than $10$ before $n\approx 250$, 
but exceeding $10$ at larger dimensions. Naturally, 
the number of iterations for the same dimension may be very different, 
depending on the configuration and complexity of the tropical polytopes 
(i.e., the solution sets of $Ax\vee c\leq Bx\vee d$). 
In the case of maximization, the number of iterations is usually below $5$. 
Note that maximization is resolved immediately if we find the greatest point of the solution set, which suggests that the maximization
problem may be simpler. 
We also remark that there is no correlation between 
the number of iterations of the bisection and Newton methods.
In particular, it is easy to construct instances with large
integers in which the number of bisection iterations becomes
arbitrarily large, whereas the number of Newton iterations
remains bounded. This agrees with Figure~\ref{f:bigentries}, where
in comparison to the graph on the left-hand side of 
Figure~\ref{f:minmax}, $M$ is equal to $500000$ instead of $500$.

\begin{figure}
\begin{center} 
\begin{tabular}{@{}c@{{}{}}c}
\begin{minipage}{0.49\linewidth}
\includegraphics[width=\linewidth]{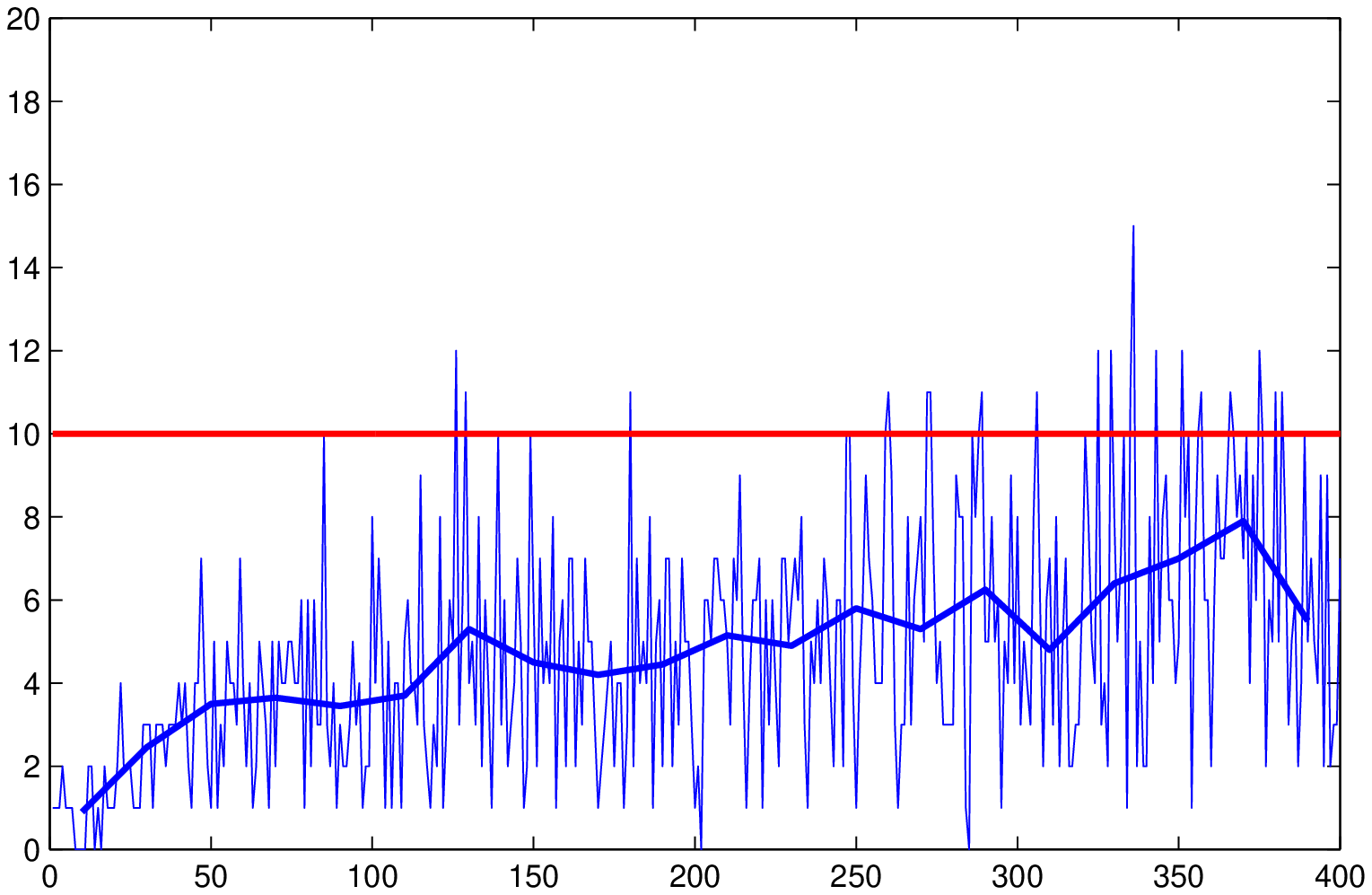}%
\end{minipage} &
\begin{minipage}{0.4\linewidth}
\includegraphics[width=\linewidth]{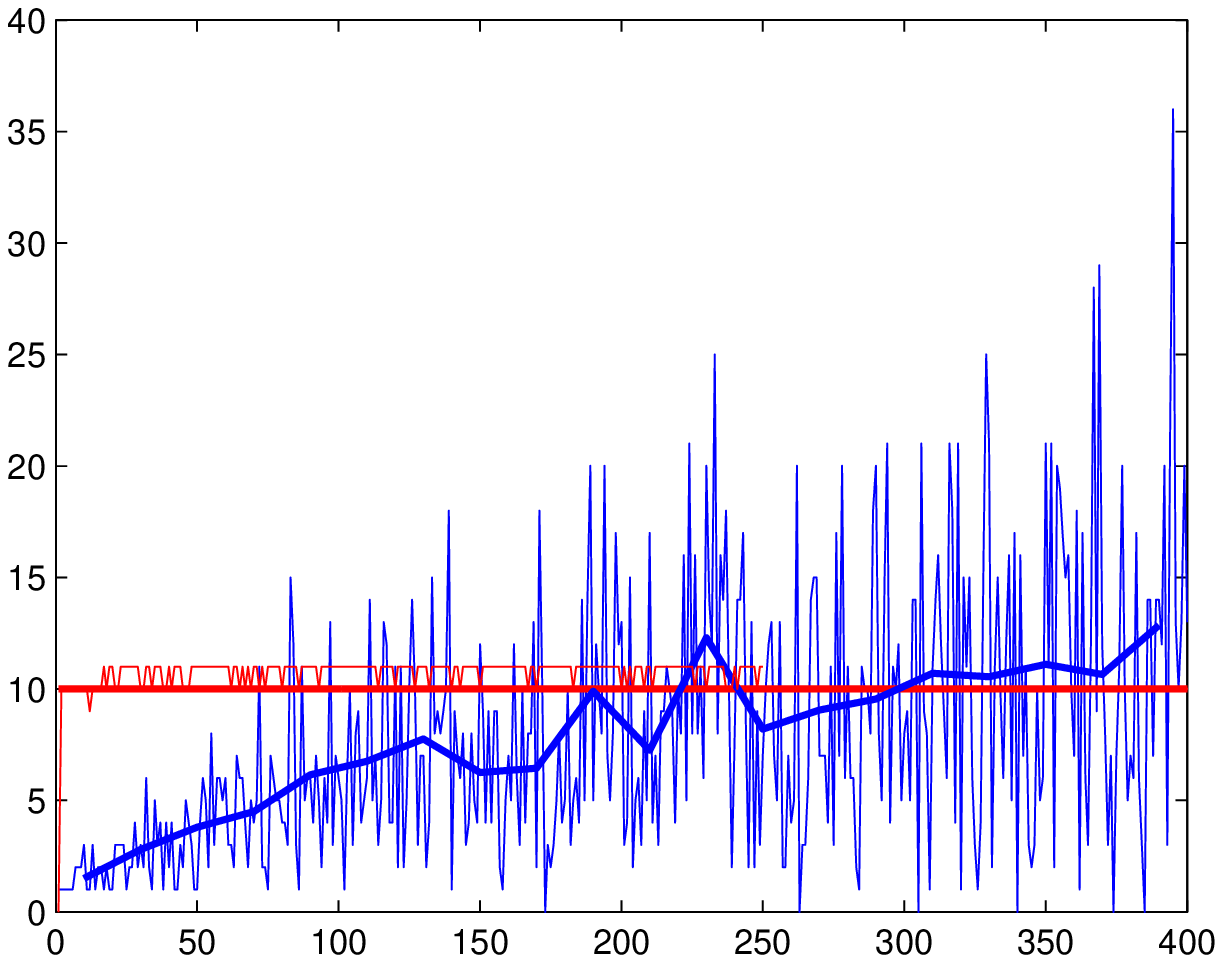}
\end{minipage}
\end{tabular}
\caption{Number of iterations of Newton method (thin blue line)
and bisection method (thin red line) in the cases of linear-fractional 
programming with finite entries (right)
and when the average proportion of $-\infty$ entries is 0.7 (left).
Thick blue line: average number of iterations of Newton method
for each interval of 20 dimensions.
Red line on the left: level $\log(2M)\approx 10$.}
\label{f:gtp}
\end{center}
\end{figure}

Figure~\ref{f:gtp} displays the cases of tropical 
linear-fractional programming~\eqref{problem-sets} with all entries finite 
(right) and with a $0.7$ frequency of $-\infty$ entries (left). 
In the case of $-\infty$ entries, 
as required by Assumptions~1 and~2, we ensure that the set of constraints 
contains neither $-\infty$ rows on the right-hand side nor $-\infty$ 
columns on the left-hand side. 
The case of tropical linear-fractional programming with finite entries 
shows almost the same picture as in the case of minimization above. 
The case when $-\infty$ appears with a regular frequency is even more favorable 
for Newton method, due to the sparsity of $\Bipdig$. 

\section{Conclusion}

In this paper, we developed an algorithm to solve tropical linear-fractional programming problems. 
This is motivated by the works~\cite{AGG08} and~\cite{AllamigeonThesis}, in which disjunctive invariants of programs are computed by tropical methods:
tropical linear-fractional programming problems are needed to tropicalize
the method 
of templates introduced
by Sankaranarayanan, Colon, Sipma and Manna~\cite{Sriram1,Sriram2}. 

The main technical ingredient, which combines ideas appearing in~\cite{AGG-10,AGK-10,GStwosided10}
 is to introduce a {\em parametric} zero-sum two-player game 
(in which the payments depend on a scalar variable), in such
a way that the value of the initial tropical linear-fractional programming problem coincides
with the smallest value of the variable for which the game is winning
for one of the players. 
The value of the parametric game, which we call the {\em spectral function},
is a piecewise affine function of the variable. Then, the problem
is reduced to finding the smallest zero of the spectral function,
which we do by a Newton-type algorithm, in which
at each iteration, we solve a one-player auxiliary game.
 
Using this game-theoretic connection, we present concise certificates expressed
in terms of the strategies of both players,  allowing one
to check whether a given feasible solution is optimal, or
whether the tropical linear-fractional programming problem is unbounded.
This is inspired by~\cite{AGK-10}, in which 
certificates of the same nature were given for the simpler problem
of certifying whether a tropical linear inequality
is a logical consequence of a finite family of such inequalities.

We also develop a generalization of the bisection method of~\cite{BA-08}. 
The latter, as well as Newton method, are shown to
be pseudo-polynomial. Note that at each iteration, both methods call an oracle 
solving a mean payoff game problem (for which the existence
of a polynomial time algorithm is an open question - only pseudo-polynomial algorithms are known). 
The pseudo-polynomial bound that we give for Newton
method is worse than the one concerning the bisection method, however,
for the former we also give a non pseudo-polynomial bound, involving the number
of strategies, which is better than the pseudo-polynomial bound
if the integers of the instance are very large. This is confirmed
by experiments, with a preliminary implementation, which indicates
that Newton method scales better as the size of the integers
grows. In addition, it has the advantage of maintaining feasibility,
and there are significant special instances in which it converges
in very few iterations. 

The Newton method of this paper appears as a natural product of the game-theoretic connection and the 
spectral function approach.
We further concentrate on the implementation of each Newton step by reduction to optimal path algorithms, and on the proof
of pseudo-polynomiality. This method could be 
also considered in the framework of more abstract Newton methods in a generalized domain, which also means 
making decent comparison with other Newton schemes, like~\cite{esparza:approximative}.
The comparison of Newton and bisection methods, as well as possible alternative approaches, also remain to be further examined. 

\paragraph{Acknowledgement}
The authors thank the anonymous reviewers for numerous important suggestions, which helped us to improve the presentation in this paper. The authors are also grateful to Peter Butkovi\v{c} for many useful discussions
concerning tropical linear programming
and tropical linear algebra. The first author thanks Xavier Allamigeon
and \'Eric Goubault for having shared with him their insights
on disjunctive invariants and static analysis.

\newcommand{\etalchar}[1]{$^{#1}$}


\begin{thebibliography}{CTCG{\etalchar{+}}98}

\bibitem[AGG08]{AGG08}
X.~Allamigeon, S.~Gaubert, and \'E. Goubault.
\newblock Inferring min and max invariants using max-plus polyhedra.
\newblock In {\em Proceedings of the 15th International Static Analysis
  Symposium (SAS'08)}, volume 5079 of {\em Lecture Notes in Comput. Sci.},
  pages 189--204. Springer, Valencia, Spain, 2008.

\bibitem[AGG09]{AGG-10}
M.~Akian, S.~Gaubert, and A.~Guterman.
\newblock Tropical polyhedra are equivalent to mean payoff games.
\newblock To appear in Int. J. of Algebra and Computation, E-print
  \arxiv{0912.2462}, 2009.

\bibitem[AGG10a]{AssaleGG}
A.~Adj\'e, S.~Gaubert, and E.~Goubault.
\newblock Coupling policy iteration with semi-definite relaxation to compute
  accurate numerical invariants in static analysis.
\newblock In A.~D. Gordon, editor, {\em Programming Languages and Systems, 19th
  European Symposium on Programming, ESOP 2010}, number 6012 in Lecture Notes
  in Comput. Sci., pages 23--42. Springer, 2010.

\bibitem[AGG10b]{AGG10}
X.~Allamigeon, S.~Gaubert, and {\'E}.~Goubault.
\newblock The tropical double description method.
\newblock In J.-Y. Marion and Th. Schwentick, editors, {\em Proceedings of the
  27th International Symposium on Theoretical Aspects of Computer Science
  (STACS 2010)}, volume~5 of {\em Leibniz International Proceedings in
  Informatics (LIPIcs)}, pages 47--58, Dagstuhl, Germany, 2010. Schloss
  Dagstuhl--Leibniz-Zentrum fuer Informatik.

\bibitem[AGK11a]{AGK-09}
X.~Allamigeon, S.~Gaubert, and R.~D. Katz.
\newblock The number of extreme points of tropical polyhedra.
\newblock {\em J. Comb. Theory Series A}, 118:162--189, 2011.
\newblock E-print \arxiv{0906.3492}.

\bibitem[AGK11b]{AGK-10}
X.~Allamigeon, S.~Gaubert, and R.~D. Katz.
\newblock Tropical polar cones, hypergraph transversals, and mean payoff games.
\newblock {\em Linear Algebra Appl.}, 435(7):1549--1574, 2011.
\newblock E-print \arxiv{1004.2778}.

\bibitem[All09]{AllamigeonThesis}
X.~Allamigeon.
\newblock {\em Static analysis of memory manipulations by abstract
  interpretation --- {A}lgorithmics of tropical polyhedra, and application to
  abstract interpretation}.
\newblock PhD thesis, \'Ecole Polytechnique, Palaiseau, France, November 2009.
\newblock
  \url{http://www.lix.polytechnique.fr/Labo/Xavier.Allamigeon/papers/thesis.pd%
f}.

\bibitem[BA08]{BA-08}
P.~Butkovi\v{c} and A.~Aminu.
\newblock Introduction to max-linear programming.
\newblock {\em IMA Journal of Management Mathematics}, 20(3):233--249, 2008.

\bibitem[BCOQ92]{BCOQ}
F.~L. Baccelli, G.~Cohen, G.-J. Olsder, and J.-P. Quadrat.
\newblock {\em Synchronization and Linearity: an Algebra for Discrete Event
  Systems}.
\newblock Wiley, 1992.

\bibitem[BH04]{BriecHorvath04}
W.~Briec and C.~Horvath.
\newblock $\mathbb{B}$-convexity.
\newblock {\em Optimization}, 53:103--127, 2004.

\bibitem[BNRC08]{bezem2}
M.~Bezem, R.~Nieuwenhuis, and E.~Rodr\'{\i}guez-Carbonell.
\newblock The max-atom problem and its relevance.
\newblock In I.~Cervesato, H.~Veith, and A.~Voronkov, editors, {\em Proceedings
  of the 15th International Conference on Logic for Programming, Artificial
  Intelligence, and Reasoning, LPAR'08}, volume 5330 of {\em Lecture Notes in
  Comput. Sci.}, pages 47--61. Springer, 2008.

\bibitem[BNRC10]{bezemjournal}
M.~Bezem, R.~Nieuwenhuis, and E.~Rodr\'{\i}guez-Carbonell.
\newblock Hard problems in max-algebra, control theory, hypergraphs and other
  areas.
\newblock {\em Information processing letters}, 110:113--138, 2010.

\bibitem[But10]{But:10}
P.~Butkovi{\v{c}}.
\newblock {\em Max-linear systems: theory and algorithms}.
\newblock Springer, 2010.

\bibitem[BV07]{bjorklund}
H.~Bjorklund and S.~Vorobyov.
\newblock A combinatorial strongly subexponential strategy improvement
  algorithm for mean payoff games.
\newblock {\em Discrete Appl. Math.}, 155:210--229, 2007.

\bibitem[CG79]{CG:79}
R.~A. Cuninghame-Green.
\newblock {\em Minimax Algebra}, volume 166 of {\em Lecture Notes in Economics
  and Mathematical Systems}.
\newblock Springer, Berlin, 1979.

\bibitem[CGB03]{CGB-03}
R.A. Cuninghame-Green and P.~Butkovi{\v{c}}.
\newblock The equation {$A\otimes x=B\otimes y$} over (max,+).
\newblock {\em Theoretical Computer Science}, 293:3--12, 2003.

\bibitem[CGG{\etalchar{+}}05]{Policy1}
A.~Costan, S.~Gaubert, E.~Goubault, M.~Martel, and S.~Putot.
\newblock A policy iteration algorithm for computing fixed points in static
  analysis of programs.
\newblock In {\em Proceedings of the 17th International Conference on Computer
  Aided Verification (CAV'05)}, volume 3576 of {\em LNCS}, pages 462--475.
  Springer, 2005.

\bibitem[CGQ04]{CGQ-04}
G.~Cohen, S.~Gaubert, and J.~P. Quadrat.
\newblock Duality and separation theorems in idempotent semimodules.
\newblock {\em Linear Algebra Appl.}, 379:395--422, 2004.
\newblock E-print \arxiv{math.FA/0212294}.

\bibitem[CH78]{CousotHalbwachs78-POPL}
P{.} Cousot and N{.} Halbwachs.
\newblock Automatic discovery of linear restraints among variables of a
  program.
\newblock In {\em Conference Record of the Fifth Annual ACM SIGPLAN-SIGACT
  Symposium on Principles of Programming Languages}, pages 84--97, Tucson,
  Arizona, 1978. ACM Press, New York, NY.

\bibitem[CTCG{\etalchar{+}}98]{Coc-98}
J.~Cochet-Terrasson, G.~Cohen, S.~Gaubert, M.~M. Gettrick, and J.~P. Quadrat.
\newblock Numerical computation of spectral elements in max-plus algebra.
\newblock In {\em Proceedings of the IFAC conference on systems structure and
  control}, pages 699--706, IRCT, Nantes, France, 1998.

\bibitem[CTGG99]{CGG-99}
J.~Cochet-Terrasson, S.~Gaubert, and J.~Gunawardena.
\newblock A constructive fixed-point theorem for min-max functions.
\newblock {\em Dynamics and Stability of Systems}, 14(4):407--433, 1999.

\bibitem[DG06]{DG-06}
V.~Dhingra and S.~Gaubert.
\newblock How to solve large scale deterministic games with mean payoff by
  policy iteration.
\newblock In {\em Proceedings of the 1st international conference on
  Performance evaluation methodolgies and tools (VALUETOOLS)}, volume 180,
  Pisa, Italy, 2006.
\newblock article No. 12.

\bibitem[DS04]{DS-04}
M.~Develin and B.~Sturmfels.
\newblock Tropical convexity.
\newblock {\em Doc. Math.}, 9:1--27 (electronic), 2004.
\newblock E-print \arxiv{math.MG/0308254}.

\bibitem[EGKS08]{esparza:approximative}
J.~Esparza, T.~Gawlitza, S.~Kiefer, and H.~Seidl.
\newblock Approximative methods for monotone systems of min-max-polynomial
  equations.
\newblock In {\em Proceedings of the 35th international colloquium on Automata,
  Languages and Programming (ICALP'08), Part I}, pages 698--710, 2008.

\bibitem[EM79]{EM-79}
A.~Ehrenfeucht and J.~Mycielski.
\newblock {Positional strategies for mean payoff games}.
\newblock {\em International Journal of Game Theory}, 8(2):109--113, 1979.

\bibitem[Fri09]{Friedmann-AnExponentialLowerB}
O.~Friedmann.
\newblock An exponential lower bound for the parity game strategy improvement
  algorithm as we know it.
\newblock In {\em Proceedings of the Twenty-Fourth Annual IEEE Symposium on
  Logic in Computer Science (LICS 2009)}, pages 145--156. IEEE Computer Society
  Press, August 2009.

\bibitem[GG98]{gg0}
S.~Gaubert and J.~Gunawardena.
\newblock The duality theorem for min-max functions.
\newblock {\em C. R. Acad. Sci. Paris.}, 326, S\'erie I:43--48, 1998.

\bibitem[GGTZ07]{ESOP07}
S.~Gaubert, E.~Goubault, A.~Taly, and S.~Zennou.
\newblock Static analysis by policy iteration on relational domains.
\newblock In {\em Proceedings of the Sixteenth European Symposium Of
  Programming (ESOP'07)}, volume 4421 of {\em LNCS}, pages 237--252. Springer,
  2007.

\bibitem[GK09]{GK-09}
S.~Gaubert and R.~D. Katz.
\newblock The tropical analogue of polar cones.
\newblock {\em Linear Algebra Appl.}, 431(5-7):608--625, 2009.
\newblock E-print \arxiv{0805.3688}.

\bibitem[GKK88]{GKK-88}
V.~A. Gurvich, A.~V. Karzanov, and L.~G. Khachiyan.
\newblock Cyclic games and an algorithm to find minimax cycle means in directed
  graphs.
\newblock {\em USSR Computational Mathematics and Mathematical Physics},
  28(5):85--91, 1988.

\bibitem[GP97]{maxplus97}
S.~Gaubert and M.~Plus.
\newblock Methods and applications of (max,+) linear algebra.
\newblock In R.~Reischuk and M.~Morvan, editors, {\em Proceedings of the 14th
  Annual Symposium on Theoretical Aspects of Computer Science (STACS'97)},
  number 1200 in Lecture Notes in Comput. Sci., pages 261--282, L\"ubeck, March
  1997. Springer.

\bibitem[GS07]{Seidl2}
T.~Gawlitza and H.~Seidl.
\newblock Precise relational invariants through strategy iteration.
\newblock In Jacques Duparc and Thomas~A. Henzinger, editors, {\em Computer
  Science Logic, 21st International Workshop, CSL 2007, 16th Annual Conference
  of the EACSL, Lausanne, Switzerland, September 11-15, 2007, Proceedings},
  volume 4646 of {\em LNCS}, pages 23--40. Springer, 2007.

\bibitem[GS10]{GStwosided10}
S.~Gaubert and S.~Sergeev.
\newblock The level set method for the two-sided eigenproblem.
\newblock E-print \arxiv{1006.5702}, 2010.

\bibitem[HOvdW05]{HOW:05}
B.~Heidergott, G.-J. Olsder, and J.~van~der Woude.
\newblock {\em Max-plus at Work}.
\newblock Princeton Univ. Press, 2005.

\bibitem[Jos05]{joswig04}
M.~Joswig.
\newblock Tropical halfspaces.
\newblock In {\em Combinatorial and computational geometry}, volume~52 of {\em
  Math. Sci. Res. Inst. Publ.}, pages 409--431. Cambridge Univ. Press,
  Cambridge, 2005.
\newblock E-print \arxiv{math.CO/0312068}.

\bibitem[Kat07]{katz05}
R.~D. Katz.
\newblock Max-plus {$(A,B)$}-invariant spaces and control of timed discrete
  event systems.
\newblock {\em IEEE Trans. Aut. Control}, 52(2):229--241, 2007.
\newblock E-print \arxiv{math.OC/0503448}.

\bibitem[Koh80]{Koh-80}
E.~Kohlberg.
\newblock Invariant half-lines of nonexpansive piecewise-linear
  transformations.
\newblock {\em Math. Oper. Res.}, 5(3):366--372, 1980.

\bibitem[LL69]{liggettlippman}
T.~M. Liggett and S.~A. Lippman.
\newblock Stochastic games with perfect information and time average payoff.
\newblock {\em SIAM Rev.}, 11:604--607, 1969.

\bibitem[LMS01]{LMS-01}
G.~L. Litvinov, V.~P. Maslov, and G.~B. Shpiz.
\newblock Idempotent functional analysis: An algebraic approach.
\newblock {\em Math. Notes (Moscow)}, 69(5):758--797, 2001.
\newblock E-print \arxiv{math.FA/0009128}.

\bibitem[Min04]{PhDMine}
A{.} Min\'e.
\newblock {\em Weakly Relational Numerical Abstract Domains}.
\newblock PhD thesis, \'Ecole Polytechnique, Palaiseau, France, December 2004.
\newblock \url{http://www.di.ens.fr/~mine/these/these-color.pdf}.

\bibitem[MSS04]{mohring}
R.~H. M{\"o}hring, M.~Skutella, and F.~Stork.
\newblock Scheduling with {AND}/{OR} precedence constraints.
\newblock {\em SIAM J. Comput.}, 33(2):393--415 (electronic), 2004.

\bibitem[Pur95]{puri}
A.~Puri.
\newblock {\em Theory of Hybrid Systems and Discrete Event Systems}.
\newblock PhD thesis, University of Berkeley, 1995.

\bibitem[SCSM06]{Sriram2}
S.~Sankaranarayanan, M.~Colon, H.~B. Sipma, and Z.~Manna.
\newblock Efficient strongly relational polyhedral analysis.
\newblock In {\em Verification, Model Checking, and Abstract Interpretation
  {(VMCAI'06)}}, volume 3855 of {\em Lecture Notes in Comput. Sci.}, pages
  111--125, Charleston, SC, January 2006. Springer.

\bibitem[Ser10]{Ser-lastdep}
S.~Sergeev.
\newblock Mean-payoff games and parametric tropical two-sided systems.
\newblock University of Birmingham, School of Mathematics, Preprint 2010/15.
  Available online from \url{http://web.mat.bham.ac.uk/P.Butkovic/Grant.html},
  2010.

\bibitem[SSM05]{Sriram1}
S.~Sankaranarayanan, H.~B. Sipma, and Z.~Manna.
\newblock Scalable analysis of linear systems using mathematical programming.
\newblock In {\em Verification, Model Checking and Abstract Interpretation
  (VMCAI'05)}, volume 3385 of {\em Lecture Notes in Comput. Sci.}, pages
  25--41, January 2005.

\bibitem[Zim77]{Zim-77}
K.~Zimmermann.
\newblock A general separation theorem in extremal algebras.
\newblock {\em Ekonom.-Mat. Obzor (Prague)}, 13:179--201, 1977.

\bibitem[Zim81]{Zim:81}
U.~Zimmermann.
\newblock {\em Linear and combinatorial optimization in ordered algebraic
  structures}.
\newblock North-Holland, Amsterdam, 1981.

\bibitem[Zim05]{Zim-05}
K.~Zimmermann.
\newblock Solution of some max-separable optimization problems with inequality
  constraints.
\newblock In G.~Litvinov and V.~Maslov, editors, {\em Idempotent Mathematics
  and Mathematical Physics}, volume 377, pages 363--370. American Mathematical
  Society, Providence, 2005.

\bibitem[ZP96]{ZP-96}
U.~Zwick and M.~Paterson.
\newblock The complexity of mean payoff games on graphs.
\newblock {\em Theoretical Computer Science}, 158(1-2):343--359, 1996.

\end{thebibliography}
\end{document}